\documentclass{amsart}

\usepackage{amssymb}
\usepackage{hyperref}
\hypersetup{colorlinks=true, citecolor=blue}

\newtheorem{theorem}{Theorem}[section]
\newtheorem{lemma}[theorem]{Lemma}
\newtheorem{proposition}[theorem]{Proposition}
\newtheorem{corollary}[theorem]{Corollary}
\newtheorem{othertheorem}{Theorem}

\theoremstyle{definition}
\newtheorem{definition}[theorem]{Definition}
\theoremstyle{remark}
\newtheorem{remark}[theorem]{Remark}

\begin{document}

\title{$\mathcal H$-Harmonic Bergman-Besov Spaces on the Real Hyperbolic Ball}

\author{A. Ers$\dot{\hbox{\i}}$n \"Ureyen}

\address{Department of Mathematics, Faculty of Science,
Eski\c{s}ehir Technical University, 26470, Eski\c{s}ehir,
Turkey}
\email{aeureyen@eskisehir.edu.tr}

\thanks{This research is supported by Eski\c{s}ehir Technical University
Research Fund under grant 24ADP108.}

\subjclass{Primary 31C05; Secondary 46E22}

\keywords{Real hyperbolic ball, hyperbolic harmonic function,
Bergman-Besov spaces}

\date{\today}

\begin{abstract}
Using characterizations in terms of various differential
operators, including partial, normal, and tangential
derivatives, we extend the family of Bergman spaces of
$\mathcal H$-harmonic functions on the real hyperbolic ball
from the range $\alpha>-1$ to all $\alpha\in\mathbb R$.
We then generalize several properties of Bergman spaces;
projection, duality, atomic decomposition, and inclusion relations,
to this extended family.
\end{abstract}

\maketitle

\section{Introduction}

For $n\ge 2$, let $\mathbb B\subseteq\mathbb R^n$ be the unit ball and
$\mathbb S=\partial\mathbb B$ the unit sphere.
The real hyperbolic ball is $\mathbb B$ equipped with the hyperbolic metric
\begin{equation*}
ds^2=\frac{4}{(1-\vert x\vert^2)^2}\sum_{i=1}^ndx_i^2.
\end{equation*}
The Laplacian $\Delta_h$ with respect to the hyperbolic metric is given by
\begin{equation*}
\Delta_h f(a)=\Delta (f\circ\varphi_a)(0)
\qquad (f\in C^2(\mathbb B)),
\end{equation*}
where
\begin{equation*}
\varphi_a(x)=\frac{a\vert x-a\vert^2+(1-\vert a\vert^2)(a-x)}
{1-2\langle x,a\rangle+\vert a\vert^2\vert x\vert^2}
\qquad (x\in\mathbb B)
\end{equation*}
is an involutive M\"obius automorphism of $\mathbb B$ that exchanges $a$ and $0$.
A straightforward calculation shows that
\begin{equation*}
\Delta_h f(x)=(1-\vert x\vert^2)^2\Delta f(x)
+2(n-2)(1-\vert x\vert^2)\langle x,\nabla f(x)\rangle.
\end{equation*}
Here, $\Delta=\sum_{i=1}^n \partial^2/\partial x_i^2$ is the Euclidean Laplacian,
$\nabla=(\partial/\partial x_1,\ldots,\partial/\partial x_n)$ is the Euclidean
gradient, and $\langle x,y\rangle=\sum_{i=1}^n x_iy_i$ is the Euclidean inner product.

Functions that are annihilated by $\Delta_h$ are called hyperbolic harmonic or
$\mathcal H$-harmonic.
We denote the space of all $\mathcal H$-harmonic functions by $\mathcal H(\mathbb B)$,
and refer the reader to \cite{St1} for the properties of $\Delta_h$ and
$\mathcal H$-harmonic functions.

Let $\nu$ and $\sigma$ denote the volume and surface measures on $\mathbb B$
and $\mathbb S$, normalized so that $\nu(\mathbb B)=1$ and $\sigma(\mathbb S)=1$.
For $\alpha\in\mathbb R$, we define the weighted measures
\begin{equation*}
  d\nu_\alpha(x)=\frac{1}{V_\alpha}(1-\vert x\vert^2)^\alpha d\nu(x)
\end{equation*}
which are finite only when $\alpha>-1$ and in this case
we set
\begin{equation*}
V_\alpha=\frac{\Gamma(n/2+1)\Gamma(\alpha+1)}{\Gamma(n/2+\alpha+1)},
\end{equation*}
so that $\nu_\alpha(\mathbb B)=1$.
When $\alpha\le -1$, we set $V_\alpha=1$.
For $\alpha\in\mathbb R$ and $1\le p<\infty$, we denote the Lebesgue spaces
with respect to $\nu_\alpha$ by $L^p_\alpha(\mathbb B)=L^p_\alpha$.

For $\alpha>-1$ and $1\le p<\infty$, the $\mathcal H$-harmonic Bergman
space $\mathcal B^p_\alpha$ is $\mathcal H(\mathbb B)\cap L^p_\alpha$,
that is,
\begin{equation*}
\mathcal B^p_\alpha
=\bigl\{f\in\mathcal H(\mathbb B)\colon\|f\|^p_{\mathcal B^p_\alpha}
=\int_{\mathbb B}\vert f(x)\vert^p\,d\nu_\alpha(x)<\infty\bigr\}.
\end{equation*}
These are Banach spaces on which point evaluations are bounded.
We refer to \cite[Chapter 10]{St1} for their basic properties.

The first aim of this paper is to characterize $\mathcal H$-harmonic
Bergman functions in terms of various differential operators, including
partial, normal, and tangential derivatives.
This will allow us to extend the Bergman space family from $\alpha>-1$ to
all $\alpha\in\mathbb R$.
This extended family will be called the Bergman-Besov spaces.
We will then generalize some known properties of Bergman spaces to
this extended family.

For partial derivatives, we use multi-index notation.
For $\kappa=(\kappa_1,\ldots,\kappa_n)$ with $\kappa_i$ non-negative
integer, we write
\begin{equation*}
\partial^\kappa f
=\frac{\partial^{\vert\kappa\vert}f}
{\partial x_1^{\kappa_i}\cdots\partial x_n^{\kappa_n}},
\end{equation*}
where $\vert\kappa\vert=\kappa_1+\cdots+\kappa_n$.
We also write $x^\kappa=x_1^{\kappa_1}\cdots x_n^{\kappa_n}$.
The normal (radial) derivative $N$ is defined by
\begin{equation*}
  Nf(x)=\langle x,\nabla f(x)\rangle,
\end{equation*}
and for $k\ge 2$, we define $N^k f=N(N^{k-1}f)$.
For $1\le i<j\le n$, the tangential derivative $T_{i,j}$ is
defined by
\begin{equation*}
T_{i,j}f(x)
=x_i\frac{\partial f}{\partial x_j}-x_j\frac{\partial f}{\partial x_i}.
\end{equation*}
For $k\ge 1$, we denote the set of all tangential derivatives of
order $k$ by $\mathcal T^k$,
\begin{equation*}
\mathcal T^k=\{T_{i_1,j_1}\circ\dots\circ T_{i_k,j_k}\colon
1\le i_l<j_l\le n\}.
\end{equation*}
For $k=0$ we set $\mathcal T^0=\{\textup{Id.}\}$.

We will also employ another type of differential operators
denoted by $D^t_s$ ($s,t\in\mathbb R$) that act as coefficient
multipliers on the series expansion of $f\in\mathcal H(\mathbb B)$.
Let $H_m=H_m(\mathbb R^n)$ be the space of (Euclidean) harmonic
polynomials homogeneous of degree $m$.
By \cite[Theorem 6.3.1]{St1}, every $f\in\mathcal H(\mathbb B)$ has a unique
series expansion
\begin{equation}\label{fseries}
f(x)=\sum_{m=0}^\infty S_m(\vert x\vert)f_m(x)
\qquad (x\in\mathbb B),
\end{equation}
where $f_m\in H_m$ and $S_m$ is the hypergeometric function
\begin{equation}\label{Sm}
S_m(r)=\frac{{}_2F_1(m,1-\frac{n}{2};m+\frac{n}{2};r^2)}
{{}_2F_1(m,1-\frac{n}{2};m+\frac{n}{2};1)}
\qquad(0\le r\le 1)
\end{equation}
normalized so that $S_m(1)=1$.
For $f\in\mathcal H(\mathbb B)$ with series expansion \eqref{fseries},
we define
\begin{equation*}
D^t_s f(x)=\sum_{m=0}^\infty d_m(s,t)S_m(\vert x\vert)f_m(x),
\end{equation*}
where the coefficients $d_m(s,t)\in\mathbb R_{>0}$ are defined in
terms of the coefficients of the reproducing kernels as explained
in Section \ref{SRkernel}.
We have $d_m(s,t)\sim m^t$ and $D^t_s$ acts as a differential operator of
order $t$ (integral if $t<0$).
The operators $D^t_s$ are compatible with reproducing kernels and,
compared with other derivative operators, are more convenient
for studying the properties of Bergman and Bergman-Besov spaces.

For $\mathcal H$-harmonic Bergman functions, the following characterization holds.

\begin{theorem}\label{TBergman}
Let $\alpha>-1$, $1\le p<\infty$ and $f\in\mathcal H(\mathbb B)$.
The following are equivalent:
\begin{enumerate}
  \item[(a)] $f\in \mathcal B^p_\alpha$.
  \item[(b)] For some (equivalently, all) $k\ge 1$, we have
  $T^k f\in \mathcal B^p_{\alpha+pk}$ for all $T^k\in\mathcal T^k$.
  \item[(c)] For some (equivalently, all) $k\ge 1$, we have
  $\partial^\kappa f\in L^p_{\alpha+pk}$ for all
  $\vert\kappa\vert=k$.
  \item[(d)] For some (equivalently, all) $k\ge 1$, we have
  $N^k f\in L^p_{\alpha+pk}$.
  \item[(e)] For some (equivalently, all) $t,s\in\mathbb R$ with $\alpha+pt>-1$,
  we have $D^t_s f\in \mathcal B^p_{\alpha+pt}$.
\end{enumerate}
Moreover, all the induced norms are equivalent.
That is,
\begin{align*}
\|f\|_{\mathcal B^p_\alpha}
&\sim \sum_{T^k\in\mathcal T^k} \|T^k f\|_{L^p_{\alpha+pk}}+\vert f(0)\vert
\sim \sum_{\vert\kappa\vert=k} \|\partial^\kappa f\|_{L^p_{\alpha+pk}}
     +\sum_{\vert\kappa\vert\le k-1} \vert\partial^\kappa f(0)\vert\\
&\sim \|N^k f\|_{L^p_{\alpha+pk}}+\vert f(0)\vert
\sim \|D^t_s f\|_{L^p_{\alpha+pt}}.
\end{align*}
\end{theorem}

This theorem is mostly known.
The equivalence (a)$\Leftrightarrow$(b)$\Leftrightarrow$(d) is shown
in \cite[Theorem 1.2]{J}.
For $k=1$, the equivalence (a)$\Leftrightarrow$(c) is \cite[Theorem 10.3.3]{St1}.
In \cite{RKSL}, a result similar to (a)$\Leftrightarrow$(e) is proved,
where different coefficient multipliers are used.
We will give an alternative proof of Theorem \ref{TBergman} that differs
from these sources.
It will be based on reproducing formulas and estimates of reproducing kernels.
For an analogous characterization of Euclidean harmonic Bergman functions,
see \cite{CKY}.

The main purpose of this paper is to extend the above result to
the range $\alpha\le -1$.
We note that if $f\in\mathcal H(\mathbb B)$, then $T_{i,j}f$ and $D^t_s f$ are also
in $\mathcal H(\mathbb B)$.
However, this is not true for partial and normal derivatives.
Correspondingly, while Theorem \ref{TBergman} extends to all
$\alpha\in\mathbb R$ for tangential derivatives and the operators $D^t_s$,
for partial and normal derivatives it is necessary to restrict $\alpha$.

\begin{theorem}\label{TBesov}
Let $\alpha\in\mathbb R$, $1\le p<\infty$ and $f\in\mathcal H(\mathbb B)$.
The following are equivalent:
\begin{enumerate}
  \item[(a)] For some (equivalently, all) $k\ge 0$ with $\alpha+pk>-1$,
  we have $T^k f\in \mathcal B^p_{\alpha+pk}$ for all $T^k\in\mathcal T^k$.
  \item[(b)] For some (equivalently, all) $t,s\in\mathbb R$ with $\alpha+pt>-1$,
  we have $D^t_s f\in \mathcal B^p_{\alpha+pt}$.
\end{enumerate}
Moreover,
\begin{equation*}
\sum_{T^k\in\mathcal T^k} \|T^k f\|_{L^p_{\alpha+pk}}+\vert f(0)\vert
\sim \|D^t_s f\|_{L^p_{\alpha+pt}}
\end{equation*}
for every $k\ge 0$ with $\alpha+pk>-1$, and for every $t,s\in\mathbb R$
with $\alpha+pt>-1$.
\end{theorem}

Using this theorem, we extend the spaces $\mathcal B^p_\alpha$ from $\alpha>-1$
to all $\alpha\in\mathbb R$.

\begin{definition}\label{DBesov}
For $\alpha\in\mathbb R$ and $1\le p<\infty$, the Bergman-Besov space
$\mathcal B^p_\alpha$ consists of all $f\in\mathcal H(\mathbb B)$ that satisfy
the equivalent conditions given in Theorem \ref{TBesov}.
\end{definition}

For every $k$ with $\alpha+pk>-1$,
$\sum_{T^k\in\mathcal T^k} \|T^k f\|_{L^p_{\alpha+pk}}+\vert f(0)\vert$
is a norm on $\mathcal B^p_\alpha$.
Similarly, for  every $t,s\in\mathbb R$ with $\alpha+pt>-1$,
$\|D^t_s f\|_{L^p_{\alpha+pt}}$ is a norm.
All these norms are equivalent, and we denote any one of these
by $\|\cdot\|_{\mathcal  B^p_\alpha}$.

If $\alpha>-1$, then one can take $k=0$ (or $t=0$) and obtain the Bergman spaces.
Special cases of the Besov zone $\alpha\le -1$ have been investigated before.
When $\alpha=-1$ and $p=2$, $\mathcal B^2_{-1}$ is the Hardy space
$\mathcal H^2$ which is studied in \cite{St2}.
When $\alpha>-p-1$, first order derivatives suffice and
$\mathcal B^p_\alpha$ coincides with the Dirichlet type spaces
$\mathcal D^p_\gamma$ (with $\gamma=\alpha+n$) considered
in \cite[Chapter 10]{St1}.
The case $\alpha=-n$ is especially important as the measure
$d\nu_{-n}(x)=(1-\vert x\vert^2)^{-n}d\nu(x)$ is M\"obius invariant.
When $p=2$, the space $\mathcal B^2_{-n}$ is extensively studied in
\cite{BEY} and \cite{BE}, where it is shown that $\mathcal B^2_{-n}$ is the unique
M\"obius invariant Hilbert space of $\mathcal H$-harmonic functions.
It is true that the spaces $\mathcal B^p_{-n}$ are M\"obius
invariant for all $1\le p<\infty$, but this will be considered in a
different work.

We next consider characterizations of Bergman-Besov functions in terms
of partial and normal derivatives.
This requires a restriction on $\alpha$.

\begin{theorem}\label{TBesovPartial}
Let $1\le p<\infty$ and $\alpha\in\mathbb R$ satisfy
\begin{equation}\label{Condalpha}
\alpha+p(n-1)>-1.
\end{equation}
For $f\in\mathcal H(\mathbb B)$, the following are equivalent:
\begin{enumerate}
  \item[(a)] $f\in \mathcal B^p_\alpha$.
  \item[(b)] For some (equivalently, all) $k\ge 0$ with $\alpha+pk>-1$,
  we have $\partial^\kappa f\in L^p_{\alpha+pk}$ for all
  $\vert\kappa\vert=k$.
  \item[(c)] For some (equivalently, all) $k\ge 0$ with $\alpha+pk>-1$,
  we have $N^k f\in L^p_{\alpha+pk}$.
\end{enumerate}
In addition,
\begin{equation*}
\|f\|_{\mathcal B^p_\alpha}
\sim \sum_{\vert\kappa\vert=k} \|\partial^\kappa f\|_{L^p_{\alpha+pk}}
     +\sum_{\vert\kappa\vert\le k-1} \vert\partial^\kappa f(0)\vert
\sim \|N^k f\|_{L^p_{\alpha+pk}}+\vert f(0)\vert.
\end{equation*}
\end{theorem}

We note that \eqref{Condalpha} is needed only for parts
(a)$\Rightarrow$(b) and (a)$\Rightarrow$(c), but not for the reverse
implications.
That is, for all $\alpha\in\mathbb R$, if there exists a $k$ with
$\alpha+pk>-1$ such that $\partial^\kappa f\in L^p_{\alpha+pk}$ for all
$\vert\kappa\vert=k$ (or $N^k f\in L^p_{\alpha+pk}$),
then $f\in\mathcal B^p_\alpha$.

For the M\"obius invariant space $\mathcal B^2_{-n}$,
Theorem \ref{TBesovPartial} is proved in \cite[Theorem 14]{BEY}.
However, \cite{BEY} imposes the additional condition $k\le n-2$.
The $k$-th derivative of the hypergeometric function
$S_m(r)$ in \eqref{Sm} is bounded on $r\in[0,1)$ for $k\le n-2$, but
becomes unbounded when $k$ exceeds $n-2$.
This leads to a change in the behavior of the $k$-th order partial
derivatives of $\mathcal H$-harmonic functions.
For an example, see the estimates of the derivatives of
the reproducing kernels given in Theorem \ref{TPartialKernel} below.
Theorem \ref{TBesovPartial} shows that for
characterizations of Bergman-Besov functions, no restriction on
the order of the derivative is needed.

We next extend several known properties (projection, duality,
atomic decomposition, inclusion relations) of $\mathcal H$-harmonic
Bergman spaces to the whole range $\alpha\in\mathbb R$.

\begin{theorem}\label{TIso}
Let $1\le p<\infty$.
$\mathcal B^p_{\alpha_1}$ is isomorphic to $\mathcal B^p_{\alpha_2}$
for all $\alpha_1,\alpha_2\in\mathbb R$.
Indeed, for any $s\in\mathbb R$ and $t=(\alpha_2-\alpha_1)/p$,
$D^t_s\colon\mathcal B^p_{\alpha_1}\to\mathcal B^p_{\alpha_2}$ is an
isomorphism.
Thus, all $\mathcal B^p_\alpha$ are Banach spaces.
\end{theorem}

Let $\mathcal R_\alpha(x,y)$ be the reproducing kernel of $\mathcal B^2_\alpha$,
$\alpha>-1$.
We will review the properties of these kernels and extend them
to $\alpha\in\mathbb R$ in Section \ref{SRkernel}.
For $\alpha\in\mathbb R$ and $\varphi\in L^1_\alpha$, define the projection
operator
\begin{equation*}
P_\alpha \varphi(x)=\int_{\mathbb B}\mathcal R_\alpha(x,y)\varphi(y)
\,d\nu_\alpha(y).
\end{equation*}
For Bergman spaces, the projection theorem below is proved
in \cite[Theorem 1.1]{U}

\begin{othertheorem}\label{TProjBergman}
Let $\alpha,\beta>-1$ and $1\le p<\infty$.
Then $P_\beta\colon L^p_\alpha\to\mathcal B^p_\alpha$ is bounded if and
only if $\alpha+1<p(\beta+1)$.
Under this condition $P_\beta f=f$ for $f\in\mathcal B^p_\alpha$.
\end{othertheorem}

The following theorem generalizes the above theorem to
all $\alpha\in\mathbb R$.

\begin{theorem}\label{TProjBesov}
Let $\alpha, \beta\in\mathbb R$ and $1\le p<\infty$.
Then $P_\beta\colon L^p_\alpha\to\mathcal B^p_\alpha$ is bounded if and only if
\begin{equation}\label{projineq}
\alpha+1<p(\beta+1).
\end{equation}
Under this condition, $P_\beta$ is right invertible.
More specifically, pick any $t\in\mathbb R$ with $\alpha+pt>-1$.
For $f\in\mathcal B^p_\alpha$, let
$\varphi(x)=\frac{V_\beta}{V_{\beta+t}}(1-\vert x\vert^2)^tD^t_\beta f(x)$.
Then $\|\varphi\|_{L^p_\alpha}\sim\|f\|_{\mathcal B^p_\alpha}$ and
$P_\beta\varphi=f$.
Thus, the following integral representation holds
\begin{equation}\label{IntRep}
f(x)=\int_{\mathbb B}
\mathcal R_\beta(x,y) D^t_\beta f(y)\,d\nu_{\beta+t}(y)
\qquad (f\in\mathcal B^p_\alpha,\, x\in\mathbb B).
\end{equation}
\end{theorem}

\begin{corollary}\label{CPointeva}
Point evaluations are bounded on $\mathcal B^p_\alpha$ for all
$\alpha\in\mathbb R$ and $1\le p<\infty$.
\end{corollary}

Next, we next consider dual spaces and generalize \cite[Corollary 1.4]{U} and
\cite[Theorem 1.2]{U3}.

\begin{theorem}\label{TDuality}
Let $\alpha\in\mathbb R$.
For $1<p<\infty$, the dual of $\mathcal B^p_\alpha$ is isomorphic
to $\mathcal B^{p'}_\alpha$, where $p'=p/(p-1)$ is the conjugate
exponent of $p$.
More precisely, pick any $s,t\in\mathbb R$ with $\alpha+pt>-1$.
Set $t'=(p-1)t$ so that $\alpha+p't'=\alpha+pt$.
Define the pairing $\langle\cdot,\cdot\rangle_{\alpha,p,s,t}
=\langle\cdot,\cdot\rangle_{\alpha,p}$ by
\begin{equation*}
\langle f,g\rangle_{\alpha,p}
=\int_{\mathbb B} D^t_s f(x) D^{t'}_sg(x)\,d\nu_{\alpha+pt}(x).
\end{equation*}
Then, to each $\Lambda\in(\mathcal B^p_\alpha)^*$, there corresponds
a unique $g\in\mathcal B^{p'}_\alpha$ with
$\|g\|_{\mathcal B^{p'}_\alpha}\sim \|\Lambda\|$ such that
$\Lambda (f)=\langle f,g\rangle_{\alpha,p}$.

The dual of $\mathcal B^1_\alpha$ is isomorphic to the
$\mathcal H$-harmonic Bloch space
\begin{equation*}
\mathcal B=\{f\in\mathcal H(\mathbb B)\colon
\sup_{x\in\mathbb B} (1-\lvert x\rvert^2)\lvert\nabla f(x)\rvert<\infty\}.
\end{equation*}
More precisely, pick any $s,t\in\mathbb R$ with $\alpha+t>-1$ and define
the pairing $\langle\cdot,\cdot\rangle_{\alpha,s,t}
=\langle\cdot,\cdot\rangle_{\alpha}$ by
\begin{equation}\label{pairing}
\langle f,g\rangle_\alpha
=\lim_{r\to1^-}\int_{r\mathbb B} D^t_s f(x) g(x)\,d\nu_{\alpha+t}(x).
\end{equation}
Then, to each $\Lambda\in(\mathcal B^1_\alpha)^*$, there corresponds
a unique $g\in\mathcal B$ with
$\|g\|_{\mathcal B}\sim \|\Lambda\|$ such that
$\Lambda (f)=\langle f,g\rangle_{\alpha}$.

Similarly, the predual of $\mathcal B^1_\alpha$ can be identified
with the $\mathcal H$-harmonic little Bloch space
$\mathcal B_0=\{f\in\mathcal H(\mathbb B)\colon
\lim_{\lvert x\rvert\to 1^-}
(1-\lvert x\rvert^2)\lvert\nabla f(x)\rvert=0\}$
under the pairing \eqref{pairing}.
\end{theorem}

We next consider atomic decompositions of $\mathcal H$-harmonic
Bergman-Besov spaces, which allow every $f\in\mathcal B^p_\alpha$
to be represented as a weighted sum of reproducing kernels.
For $a,b\in\mathbb B$, let $\rho(a,b)=\vert\varphi_a(b)\vert$
be the pseudo-hyperbolic metric, and for $0<r<1$, let
$E_r(a)=\{x\in\mathbb B\colon\rho(x,a)<r\}$ be the pseudo-hyperbolic
ball of radius $r$ centered at $a$.
A sequence $\{a_m\}_{m=1}^\infty$ of points of $\mathbb B$ is called
$r$-separated if $\rho(a_k,a_m)\ge r$ for $k\neq m$.
An $r$-separated sequence is called an $r$-lattice if
$\cup_{m=1}^\infty E_r(a_m)=\mathbb B$.
The following theorem extends the atomic decomposition of Bergman spaces
(see \cite[Theorem 1.1 with Eqn.~(5)]{U2}) to $\alpha\in\mathbb R$.

\begin{theorem}\label{TAtomic}
For $\alpha\in\mathbb R$ and $1\le p<\infty$, let $s$ satisfy
$\alpha+1<p(s+1)$.
There exists an $r_0<1/8$ depending only on $n,\alpha,p,s$ such that
if $\{a_m\}$ is an $r$-lattice with $r<r_0$, then for every
$f\in\mathcal B^p_\alpha$, there exists $\{\lambda_m\}\in\ell^p$ with
$\|\{\lambda_m\}\|_{\ell^p}\sim\|f\|_{\mathcal B^p_\alpha}$ such that
\begin{equation*}
f(x)=\sum_{m=1}^\infty\lambda_m(1-\vert a_m\vert^2)^{(s+n)-(\alpha+n)/p}
\mathcal R_s(x,a_m)
\qquad (x\in\mathbb B),
\end{equation*}
where the series converges absolutely and uniformly on compact subsets
of $\mathbb B$ and also in $\mathcal B^p_\alpha$.
\end{theorem}

The next theorem extends the inclusion relations of \cite[Theorem 1.3]{U2}.

\begin{theorem}\label{TInclusion}
Let $\alpha,\beta\in\mathbb R$ and $1\le p,q<\infty$.
\begin{enumerate}
\item[(a)] If $q\geq p$, then
\begin{equation*}
  \mathcal B^p_\alpha\subset\mathcal B^q_\beta\quad
  \text{if and only if}\quad
  \frac{\alpha+n}{p}\leq\frac{\beta+n}{q}.
\end{equation*}
\item[(b)] If $q<p$, then
\begin{equation*}
  \mathcal B^p_\alpha\subset\mathcal B^q_\beta\quad
  \text{if and only if}\quad
  \frac{\alpha+1}{p}<\frac{\beta+1}{q}.
\end{equation*}
\end{enumerate}
In both cases, the inclusion
$i\colon\mathcal B^p_\alpha\to\mathcal B^q_\beta$ is continuous.
\end{theorem}

\section{Preliminaries}

We denote positive constants whose exact values are inessential by
the letter $C$.
For two positive expressions $X$ and $Y$, we write $X\lesssim Y$
to mean $X\leq C Y$.
If both $X\leq CY$ and $Y\leq CX$, we write $X\sim Y$.

For $x,y\in\mathbb B$, we write
\begin{equation*}
  [x,y]:=\sqrt{1-2\langle x,y\rangle+\vert x\vert^2\vert y\vert^2}.
\end{equation*}
Clearly, $[x,y]$ is symmetric; when $y=0$, $[x,0]=1$ and otherwise
\begin{equation}\label{xyis}
[x,y]=\bigl\vert \vert y\vert x-\frac{y}{\vert y\vert}\bigr\vert.
\end{equation}

\begin{lemma}\label{Ltxy}
For all $x,y\in\mathbb B$ and $0\le t\le 1$, we have
\begin{equation*}
[tx,y]\ge\frac{[x,y]}{2}.
\end{equation*}
\end{lemma}

\begin{proof}
By \eqref{xyis}, it suffices to show that
$\vert tx-\zeta\vert\ge\vert x-\zeta\vert/2$ for all
$x\in\mathbb B$, $\zeta\in\mathbb S$ and $0\le t\le 1$.
This follows from the triangle inequality
$\vert x-\zeta\vert\le \vert x-tx\vert +\vert tx-\zeta\vert$ and the fact that
$\vert x-tx\vert=\vert x\vert-t\vert x\vert\le 1-t\vert x\vert\le \vert tx-\zeta\vert$.
\end{proof}

The following lemma follows from \cite[Lemma 2.15]{U} and \eqref{xyis}.
\begin{lemma}\label{LInt01}
Let $b>-1$ and $c\in\mathbb R$.
There exists $C=C(n,b,c)$ such that for all $x,y\in\mathbb B$,
\begin{equation*}
\int_0^1
\frac{(1-t)^b}{[tx,y]^{1+b+c}}\,dt
\le C
\begin{cases}
\dfrac{1}{[x,y]^{c}},&\text{if $c>0$};\\
1+\log\dfrac{1}{[x,y]},&\text{if $c=0$};\\
1,&\text{if $c<0$}.
\end{cases}
\end{equation*}
\end{lemma}

For $s,t\in\mathbb R$, define the operator
\begin{equation*}
E_{s,t} f(x)=(1-\vert x\vert^2)^t
\int_{\mathbb B} f(y)\frac{(1-\vert y\vert^2)^s}{[x,y]^{n+s+t}}\,d\nu(y).
\end{equation*}
For a proof of the following lemma, see, for example, \cite[Theorem 1.6]{GKU}.

\begin{lemma}\label{LEbdd}
Let $\alpha,s,t\in\mathbb R$ and $1\le p<\infty$.
Then $E_{s,t}\colon L^p_\alpha\to L^p_\alpha$ is bounded
if and only if $-pt<\alpha+1<p(s+1)$.
\end{lemma}

\section{Reproducing Kernels and the Operators $D^t_s$}\label{SRkernel}

For $\alpha>-1$, point evaluations are bounded on the Hilbert space
$\mathcal B^2_\alpha$ and so, for each $x\in\mathbb B$, there exists
$\mathcal R_\alpha(x,\cdot)\in\mathcal B^2_\alpha$ such that
\begin{equation*}
f(x)=\int_\mathbb B f(y)\overline{\mathcal R_\alpha(x,y)}\,d\nu_\alpha(y)
\qquad (f\in\mathcal B^2_\alpha).
\end{equation*}
The reproducing kernel $\mathcal R_\alpha$ is real valued (so
conjugation can be deleted), symmetric, and $\mathcal H$-harmonic
in each variable.
It has the series expansion
\begin{equation}\label{seriesRalpha}
\mathcal R_\alpha(x,y)=\sum_{m=0}^\infty
c_m(\alpha)S_m(\vert x\vert)S_m(\vert y\vert)Z_m(x,y),
\end{equation}
where the coefficient $c_m(\alpha)$ is determined by
\begin{equation*}
\frac{1}{c_m(\alpha)}
=\frac{n}{V_\alpha}
\int_0^1 r^{2m+n-1}(1-r^2)^\alpha S_m^2(r)\,dr.
\end{equation*}
Here, $Z_m(\cdot,\cdot)$, the zonal harmonic of degree $m$, is the
reproducing kernel of $H_m(\mathbb S)$ extended to
$\mathbb R^n\times\mathbb R^n$ by homogeneity (see \cite[Chapter 5]{ABR}).
The series in \eqref{seriesRalpha} converges absolutely and uniformly on
$K\times\mathbb B$ for any compact $K\subset\mathbb B$.

In (\cite[Corollary 3.4]{U}), it is shown that there exist constants
$D_k=D_k(\alpha,n)$ with $D_0>0$ such that the asymptotic expansion
\begin{equation}\label{cmbig}
c_m(\alpha)\approx\frac{\Gamma(m+\alpha+n)}{\Gamma(m+n-1)}
\sum_{k=0}^\infty \frac{D_k}{m^k}
\qquad (m\to\infty),
\end{equation}
holds.
More precisely, for each $K\ge 1$, we have
\begin{equation*}
c_m(\alpha)=\frac{\Gamma(m+\alpha+n)}{\Gamma(m+n-1)}
\biggl(\,\sum_{k=0}^{K-1}\frac{D_k}{m^k}+O\Bigl(\frac{1}{m^K}\Bigr)\biggr)
\qquad (m\to\infty).
\end{equation*}
Thus, by Stirling's formula, $c_m(\alpha)\sim m^{\alpha+1}$ for $m\ge 1$.

Integrating in polar coordinates shows that for
$f=\sum_{m=0}^\infty S_m(\vert x\vert)f_m$,
$g=\sum_{m=0}^\infty S_m(\vert x\vert)g_m$,
\begin{equation*}
\langle f,g\rangle_{\mathcal B^2_\alpha}
=\int_{\mathbb B}f\bar{g}\,d\nu_\alpha
=\sum_{m=0}^\infty\frac{1}{c_m(\alpha)}\langle f_m,g_m\rangle_{L^2(\mathbb S)},
\end{equation*}
and so, for $\alpha>-1$,
\begin{equation}\label{B2alpha}
\mathcal B^2_\alpha=\Bigl\{f=\sum_{m=0}^\infty
S_m(\vert x\vert)f_m\in\mathcal H(\mathbb B)\colon
\sum_{m=1}^\infty\frac{1}{m^{\alpha+1}}\|f_m\|^2_{L^2(\mathbb S)}<\infty\Bigr\}.
\end{equation}
In the above characterization one can remove the assumption $\alpha>-1$ and
extend $\mathcal B^2_\alpha$ to all $\alpha\in\mathbb R$.
We show later in Remark \ref{Rp2} that this agrees with Definition \ref{DBesov}.

For $\alpha>-1$, $\mathcal B^2_\alpha$ has a natural inner product and a
corresponding reproducing kernel.
On the other hand, for $\alpha\le -1$, derivatives must be used, but
there is no natural choice for their type or order.
Here, we follow the approach of \cite{K} where holomorphic (Bergman-)Besov
spaces on the unit ball of $\mathbb C^n$ are considered.
We first define $\mathcal R_\alpha$ for $\alpha\le -1$ and
call these functions as reproducing kernels, although only later
(Remark \ref{Rp2}) we verify that these are indeed reproducing kernels
of $\mathcal B^2_\alpha$ with respect to a suitable inner product.

We first mention two special cases that have been studied in detail before.
When $\alpha=-1$, $\mathcal B^2_{-1}=\mathcal H^2$ is the Hardy space whose
reproducing kernel is the extended hyperbolic Poisson kernel
\begin{equation*}
\mathcal R_{-1}(x,y)=\sum_{m=0}^\infty
S_m(\vert x\vert)S_m(\vert y\vert)Z_m(x,y)
\end{equation*}
(see \cite[Theorem 2.5]{St2}).
When $\alpha=-n$, $\mathcal B^2_{-n}$ is M\"{o}bius invariant
and with respect to the M\"{o}bius invariant (semi-)inner product, the
reproducing kernel is (see \cite{BEY})
\begin{equation*}
\mathcal R_{-n}(x,y)=1+\sum_{m=1}^\infty
\frac{\Gamma(m)}{\Gamma(m+n-1)}S_m(\vert x\vert)S_m(\vert y\vert)Z_m(x,y).
\end{equation*}

\begin{definition}\label{Dkernel}
For $\alpha\le -1$ and $x,y\in\mathbb B$, we define
\begin{equation*}
\mathcal R_{\alpha}(x,y)=\sum_{m=0}^\infty c_m(\alpha)
S_m(\vert x\vert)S_m(\vert y\vert)Z_m(x,y),
\end{equation*}
where $c_0(\alpha)=1$ and
\begin{equation}\label{cmsmall}
c_m(\alpha)=\frac{\Gamma(m)}{\Gamma(m-\alpha-1)},
\qquad m\ge 1.
\end{equation}
\end{definition}
The only properties of $c_m(\alpha)$ we will employ later are
$c_m(\alpha)>0$ and
\begin{equation}\label{cmasym}
c_m(\alpha)\sim m^{\alpha+1}
\qquad (m\ge 1),
\end{equation}
for all $\alpha\in\mathbb R$.
Other than these, the specific formula \eqref{cmsmall} will not have
any future role and the only reason for its choice is to incorporate
the two special cases mentioned above.

Using the above coefficients $c_m(\alpha)$ $(\alpha\in\mathbb R)$,
we define differential operators $D^t_s$ which act as coefficient
multipliers on the series expansion
of $f\in\mathcal H(\mathbb B)$.

\begin{definition}
Let $s,t\in\mathbb R$.
For $f\in\mathcal H(\mathbb B)$ with
$f(x)=\sum_{m=0}^\infty S_m(\vert x\vert)f_m(x)$, we define
\begin{equation*}
D^t_sf(x)=\sum_{m=0}^\infty \frac{c_m(s+t)}{c_m(s)}
S_m(\vert x\vert)f_m(x).
\end{equation*}
\end{definition}

By \eqref{cmasym}, $c_m(s+t)/c_m(s)\sim m^t$ and $D^t_s$ acts as a
differential (integral if $t<0$) operator of order $t$.
The role of the parameter $s$ is minor, and its main role is to have
exact formula in part (ii) of the lemma below.

\begin{lemma}\label{LPropDst}
For $s,t\in\mathbb R$, the operator $D^t_s$ satisfies the following
properties.
\begin{enumerate}
\item[(i)] $D^t_s:\mathcal H(\mathbb B)\to\mathcal H(\mathbb B)$
is continuous, where $\mathcal H(\mathbb B)$ is endowed with the topology
of uniform convergence on compact subsets.
\item[(ii)] $D^t_s\mathcal R_s(x,y)=\mathcal R_{s+t}(x,y)$,
acting on the variable $x$ (or $y$).
\item[(iii)] $D^{u}_{s+t}\circ D^t_s=D^{u+t}_s$.
\item[(iv)] $D^0_s=\textup{Id}.$ and $D^t_s$ has the two-sided inverse
$D^{-t}_{s+t}$, i.e., $D^{-t}_{s+t}\circ D^t_s=\textup{Id}.$ and
$D^t_s\circ D^{-t}_{s+t}=\textup{Id}.$
\end{enumerate}
\end{lemma}

Parts (ii)--(iv) follow from the definition.
For part (i) see \cite[Lemma 3.2]{U3}.

\section{Estimates of Derivatives of Reproducing Kernels}

In this section, we estimate various derivatives of the reproducing kernels.

\begin{lemma}\label{Lderzm}
Let $k\ge 1$.
There exists a constant $C=C(n,k)$ such that
for all multi-indices $\kappa$ with $\vert\kappa\vert=k$,
$m\ge k$ and $x,y\in\mathbb B$,
\begin{equation*}
\bigl\vert\partial^\kappa Z_m(x,y)\bigr\vert
\le C m^{n-2+k}\,
\vert x\vert^{m-k}\vert y\vert^m,
\end{equation*}
where differentiation is applied to the first variable.
\end{lemma}

\begin{proof}
It is shown in \cite[Theorem 4(b)]{Se} that there exists $C=C(n,k)$
such that for all $\vert\kappa\vert=k$, $m\ge k$, $p\in H_m$
and $\zeta\in\mathbb S$,
\begin{equation*}
\vert\partial^\kappa p(\zeta)\vert
\le C m^{n/2-1+k}\|p\|_{L^2(\mathbb S)}.
\end{equation*}
As
$\|Z_m(\cdot,\eta)\|^2_{L^2(\mathbb S)}=\text{dim}\, H_m\lesssim m^{n-2}$
for $\eta\in\mathbb S$, and $\partial^\kappa Z_m(\cdot,\cdot)$ is
homogeneous of degree $m-k$ in the first variable and degree $m$ in
the second, the result follows.
\end{proof}

\begin{proposition}\label{PEstW}
(i) For $a_j,b_j\ge 0$ $(j=1,2,\dots,J)$, let
\begin{equation}\label{dmis}
d_m=\frac{\Gamma(m+a_1)\cdots\Gamma(m+a_J)}{\Gamma(m+b_1)\cdots\Gamma(m+b_J)},
\end{equation}
and for $x\in\mathbb B$, $y\in\overline{\mathbb B}$, let
\begin{equation*}
W(x,y)=\sum_{m=1}^\infty d_m Z_m(x,y).
\end{equation*}
Given a multi-index $\lambda$, set
$c=(n-1)+\sum_{j=1}^J(a_j-b_j)+\vert\lambda\vert$.
There exists a constant $C$ depending only on $n,a_j,b_j,\lambda$ such that
for all $x\in\mathbb B$, $y\in\overline{\mathbb B}$,
\begin{equation*}
\bigl\vert\partial^\lambda W(x,y)\bigr\vert
\leq C\begin{cases}
\dfrac{1}{[x,y]^{c}},&\text{if $c>0$};\\
1+\log\dfrac{1}{[x,y]},&\text{if $c=0$};\\
1,&\text{if $c<0$},
\end{cases}
\end{equation*}
where differentiation is applied to the first variable.

(ii) The same estimate holds if instead of \eqref{dmis}, $d_m$ has asymptotic
expansion
\begin{equation*}
d_m\approx\frac{\Gamma(m+a_1)\cdots\Gamma(m+a_J)}{\Gamma(m+b_1)\cdots\Gamma(m+b_J)}
\sum_{k=0}^\infty\frac{A_k}{m^k}
\qquad (m\to\infty),
\end{equation*}
where $A_k\in\mathbb R$ and $A_0\neq 0$.
\end{proposition}

\begin{proof}
Part (i) is Corollary 4.3 of \cite{U}.
Part (ii) follows from part (i) by a method similar to the proof of
Lemma 5.2(ii) of \cite{U}.
The only difference is that we choose $K>c$ and use Lemma \ref{Lderzm}
instead of \cite[Lemma 5.1]{U}.
\end{proof}

\begin{proposition}\label{PEsth}
(i) For $a_j,b_j\ge 0$ $(j=1,2,\dots,J)$, let
\begin{equation}\label{dmis2}
d_m=\frac{\Gamma(m+a_1)\cdots\Gamma(m+a_J)}{\Gamma(m+b_1)\cdots\Gamma(m+b_J)},
\end{equation}
and for $x\in\mathbb B$, $y\in\overline{\mathbb B}$, let
\begin{equation*}
h(x,y)=\sum_{m=1}^\infty d_m S_m(\vert x\vert)S_m(\vert y\vert) Z_m(x,y).
\end{equation*}
Set $c=(n-1)+\sum_{j=1}^J(a_j-b_j)$.
There exists a constant $C$ depending only on $n,a_j,b_j$ such that for all
$x\in\mathbb B$, $y\in\overline{\mathbb B}$,
\begin{equation*}
\vert h(x,y)\vert\leq C
\begin{cases}
\dfrac{1}{[x,y]^{c}},&\text{if $c>0$};\\
1+\log\dfrac{1}{[x,y]},&\text{if $c=0$};\\
1,&\text{if $c<0$}.
\end{cases}
\end{equation*}

(ii) The same estimate holds if instead of \eqref{dmis2}, $d_m$ has asymptotic
expansion
\begin{equation}\label{dm2asym}
d_m\approx\frac{\Gamma(m+a_1)\cdots\Gamma(m+a_J)}{\Gamma(m+b_1)\cdots\Gamma(m+b_J)}
\sum_{k=0}^\infty\frac{A_k}{m^k}
\qquad (m\to\infty),
\end{equation}
where $A_k\in\mathbb R$ and $A_0\neq 0$.
\end{proposition}

\begin{proof}
Part (i) is Theorem 4.4 of \cite{U}.
Part (ii) follows from part (i) in the same way as in the proof of
Theorem 1.2(a) of \cite{U}.
The only difference is that we choose $K>c+n-2$.
\end{proof}

Next, we estimate $D^t_s\mathcal R_\alpha(x,y)$.
Note that when $t=0$, $D^0_s=\textup{Id}.$ and the theorem below gives an
estimate of $\mathcal R_\alpha$.

\begin{theorem}\label{TDstKernel}
Let $\alpha,s,t\in\mathbb R$.
There exists a constant $C=C(n,\alpha,s,t)$ such that for all $x,y\in\mathbb B$,
\begin{equation*}
\vert D^t_s\mathcal R_\alpha(x,y)\vert\le C
\begin{cases}
\dfrac{1}{[x,y]^{n+\alpha+t}},&\text{if $n+\alpha+t>0$};\\
1+\log\dfrac{1}{[x,y]},&\text{if $n+\alpha+t=0$};\\
1,&\text{if $n+\alpha+t<0$}.
\end{cases}
\end{equation*}
\end{theorem}

\begin{proof}
We have
\begin{equation*}
D^t_s\mathcal R_\alpha(x,y)=\sum_{m=0}^\infty
\frac{c_m(s+t)}{c_m(s)}c_m(\alpha)
S_m(\vert x\vert)S_m(\vert y\vert)Z_m(x,y),
\end{equation*}
where $c_m$ is determined by either \eqref{cmbig} or \eqref{cmsmall}
depending on whether $\alpha,s,s+t$ are $>-1$ or $\le -1$.
So, $d_m=c_m(s+t)c_m(\alpha)/c_m(s)$ is of the form \eqref{dmis2} or \eqref{dm2asym}.
In all possible $8$ cases, the parameter $c$ in Proposition
\ref{PEsth} is $c=(n-1)+(s+t+1)+(-s-1)+(\alpha+1)=n+\alpha+t$
and the result follows.
\end{proof}

We next estimate partial derivatives of the reproducing kernels.
For simplicity, we consider only the case $\alpha>-1$ which is
sufficient for our purposes.
The estimate takes a different form when the order of the derivative
exceeds $n-2$.

\begin{theorem}\label{TPartialKernel}
Let $\alpha>-1$ and $\kappa$ be a multi-index.
There exists a constant $C=C(n,\alpha,\kappa)$ such that for all
$x,y\in\mathbb B$,
\begin{enumerate}
\item[(i)] If $\vert\kappa\vert\le n-2$, then
\begin{equation*}
\vert\partial^\kappa\mathcal R_\alpha(x,y)\vert
\le \frac{C}{[x,y]^{n+\alpha+\vert\kappa\vert}}.
\end{equation*}
\item[(ii)] If $\vert\kappa\vert=n-1$, then
\begin{enumerate}
\item[(a)]
$\displaystyle
\bigl\vert\partial^\kappa\mathcal R_\alpha(x,y)\bigr\vert
\le C\Bigl(1+\log\frac{1}{1-\vert x\vert^2}\Bigr)
\frac{1}{[x,y]^{n+\alpha+n-1}},$
\item[(b)]
for every $0<\varepsilon<1$ (with $C$ depending also on $\varepsilon$),
\begin{equation*}
\bigl\vert\partial^\kappa\mathcal R_\alpha(x,y)\bigr\vert
\le C\frac{1}{(1-\vert x\vert^2)^{\varepsilon}}
\frac{1}{[x,y]^{n+\alpha+n-1-\varepsilon}}.
\end{equation*}
\end{enumerate}
\item[(iii)] If $\vert\kappa\vert\ge n$, then
\begin{equation*}
\bigl\vert\partial^\kappa\mathcal R_\alpha(x,y)\bigr\vert
\le C\frac{1}{(1-\vert x\vert^2)^{\vert\kappa\vert-(n-1)}}
\frac{1}{[x,y]^{n+\alpha+n-1}}.
\end{equation*}
\end{enumerate}
Here, differentiation is with respect to $x$.
\end{theorem}

\begin{proof}
The proof follows the proof of Theorem 1.2(b) of \cite{U} and we will be
brief about some of the details.
To estimate $\partial^\kappa\mathcal R_\alpha(x,y)$, it suffices to estimate
\begin{equation*}
G(x,y):=\sum_{m=1}^\infty
c_m(\alpha)\partial^\lambda S_m(\vert x\vert)S_m(\vert y\vert)
\partial^\mu Z_m(x,y),
\end{equation*}
where $\lambda+\mu=\kappa$.
Since (\cite[Eqn. (22)]{U})
\begin{equation*}
S_m(\vert x\vert)
=\frac{1}{B(m,n-1)}
\int_0^1 t^{m-1}(1-t)^{\frac{n}{2}-1}(1-\vert x\vert^2 t)^{\frac{n}{2}-1}\,dt,
\end{equation*}
we have
\begin{equation*}
\partial^\lambda S_m(\vert x\vert)
=\frac{1}{B(m,n-1)}
\int_0^1 t^{m-1}(1-t)^{\frac{n}{2}-1}P_\lambda(x,t)
(1-\vert x\vert^2 t)^{\frac{n}{2}-1-\vert\lambda\vert}\,dt,
\end{equation*}
where $P_\lambda$ is a polynomial.
Thus
\begin{align*}
G(x,y)=\int_0^1\int_0^1
&(1-t)^{\frac{n}{2}-1}P_\lambda(x,t)
(1-\vert x\vert^2t)^{\frac{n}{2}-1-\vert\lambda\vert}
(1-\tau)^{\frac{n}{2}-1}(1-\vert y\vert^2 \tau)^{\frac{n}{2}-1}\\
&\sum_{m=1}^\infty\frac{c_m(\alpha)}{B^2(m,n-1)}
t^{m-1}\tau^{m-1}\partial^\mu Z_m(x,y)\,d\tau dt.
\end{align*}
Since $\partial^\mu Z_m(x,y)$ is homogeneous of degree $m-\vert\mu\vert$ in $x$
and of degree $m$ in $y$, writing
$t^{m-1}\tau^{m-1}\partial^\mu Z_m(x,y)
=t^{\vert\mu\vert}\partial^\mu Z_m(tx,\tau y)/t\tau$, and estimating the above
series with Proposition \ref{PEstW} using \eqref{cmbig}, we obtain
\begin{equation*}
\vert G(x,y)\vert
\lesssim \int_0^1\int_0^1
\frac{(1-t)^{\frac{n}{2}-1}(1-\vert x\vert^2t)^{\frac{n}{2}-1-\vert\lambda\vert}
(1-\tau)^{\frac{n}{2}-1}(1-\vert y\vert^2 \tau)^{\frac{n}{2}-1}}
{[tx,\tau y]^{n+\alpha+2(n-1)+\vert\mu\vert}}\,d\tau dt,
\end{equation*}
where we also use $\vert P_\lambda\vert\lesssim 1$, and delete the term $1/t\tau$
which can be justified as done in the proof of \cite[Theorem 4.4]{U}.
Using the inequality $1-\vert y\vert^2\tau\le 2[tx,\tau y]$ (\cite[Eqn.~(36)]{U})
and then integrating
with respect to $\tau$, Lemma \ref{LInt01} yields
\begin{equation*}
\vert G(x,y)\vert
\lesssim \int_0^1
\frac{(1-t)^{\frac{n}{2}-1}(1-\vert x\vert^2t)^{\frac{n}{2}-1-\vert\lambda\vert}}
{[tx,y]^{n+\alpha+n-1+\vert\mu\vert}}\,dt.
\end{equation*}
Finally, using the inequality (\cite[Eqn.~(35)]{U} with $\tau=1$)
\begin{equation}\label{ineq}
1-t\le 1-\vert x\vert^2 t\le 2[tx,y]
\end{equation}
and $\vert\lambda\vert+\vert\mu\vert=\vert\kappa\vert$, we obtain
\begin{equation*}
\vert G(x,y)\vert
\lesssim \int_0^1
\frac{(1-\vert x\vert^2t)^{n-2-\vert\kappa\vert}}
{[tx,y]^{n+\alpha+n-1}}\,dt.
\end{equation*}

We now consider the three cases.

(i) If $\vert\kappa\vert\le n-2$, then using the second inequality
in \eqref{ineq}, and then Lemma \ref{LInt01} shows
\begin{equation*}
\vert G(x,y)\vert
\lesssim\frac{1}{[x,y]^{n+\alpha+\vert\kappa\vert}}.
\end{equation*}

(ii) If $\vert\kappa\vert=n-1$, then by Lemma \ref{Ltxy},
\begin{align}
\vert G(x,y)\vert
&\lesssim \int_0^1\frac{dt}{(1-\vert x\vert^2 t)[tx,y]^{n+\alpha+n-1}}
\lesssim \frac{1}{[x,y]^{n+\alpha+n-1}}\int_0^1\frac{dt}{1-\vert x\vert^2 t}
\label{partii}\\
&\lesssim \Bigl(1+\log\frac{1}{1-\vert x\vert^2}\Bigr)\frac{1}{[x,y]^{n+\alpha+n-1}}.
\notag
\end{align}
This proves part (a).
Alternatively, for $0<\varepsilon<1$, let $p=1/\varepsilon$ and apply
H\"{o}lder's inequality to the first integral in \eqref{partii}.
This gives
\begin{align*}
\vert G(x,y)\vert
&\lesssim \biggl(\int_0^1\frac{dt}{(1-\vert x\vert^2 t)^{p'}}\biggr)^{1/p'}
\biggl(\int_0^1\frac{dt}{[tx,y]^{(n+\alpha+n-1)p}}\biggr)^{1/p}\\
&\lesssim \frac{1}{(1-\vert x\vert^2)^{1-1/p'}}
\frac{1}{[x,y]^{n+\alpha+n-1-1/p}},
\end{align*}
where in the last inequality we use Lemma \ref{LInt01}.
This proves part (b).

(iii) If $\vert\kappa\vert\ge n$, then applying Lemma \ref{Ltxy} shows
\begin{align*}
\vert G(x,y)\vert
&\lesssim \frac{1}{[x,y]^{n+\alpha+n-1}}
\int_0^1 (1-\vert x\vert^2 t)^{n-2-\vert\kappa\vert}\,dt\\
&\lesssim \frac{1}{(1-\vert x\vert^2)^{\vert\kappa\vert-(n-1)}}
\frac{1}{[x,y]^{n+\alpha+n-1}},
\end{align*}
and this finishes the proof.
\end{proof}

We next consider normal derivatives.
Because $\vert N^k f(x)\vert
\le C\sum_{1\le \vert\kappa\vert\le k}\vert\partial^\kappa f(x)\vert$
for $x\in\mathbb B$, and $[x,y]\le 2$ for $x,y\in\mathbb B$, the
next theorem follows immediately from Theorem \ref{TPartialKernel}.

\begin{theorem}\label{TNormalKernel}
Let $\alpha>-1$ and $k\ge 1$.
There exists a constant $C=C(n,\alpha,k)$ such that for all
$x,y\in\mathbb B$,
\begin{enumerate}
\item[(i)] If $k\le n-2$, then
$\displaystyle \vert N^k \mathcal R_\alpha(x,y)\vert
\le \frac{C}{[x,y]^{n+\alpha+k}}$.
\item[(ii)] If $k=n-1$, then
\begin{enumerate}
\item[(a)]
$\displaystyle
\bigl\vert N^k \mathcal R_\alpha(x,y)\bigr\vert
\le C\Bigl(1+\log\frac{1}{1-\vert x\vert^2}\Bigr)
\frac{1}{[x,y]^{n+\alpha+n-1}},$
\item[(b)]
for every $0<\varepsilon<1$ (with $C$ depending also on $\varepsilon$),
\begin{equation*}
\bigl\vert N^k \mathcal R_\alpha(x,y)\bigr\vert
\le C\frac{1}{(1-\vert x\vert^2)^{\varepsilon}}
\frac{1}{[x,y]^{n+\alpha+n-1-\varepsilon}}.
\end{equation*}
\end{enumerate}
\item[(iii)] If $k\ge n$, then
$\displaystyle \bigl\vert N^k \mathcal R_\alpha(x,y)\bigr\vert
\le C\frac{1}{(1-\vert x\vert^2)^{k-(n-1)}}
\frac{1}{[x,y]^{n+\alpha+n-1}}$.
\end{enumerate}
Here, differentiation is with respect to $x$.
\end{theorem}

\section{Characterizations in terms of Tangential Derivatives and the operators $D^t_s$}

In this section, we first consider the Bergman case and prove
(a)$\Leftrightarrow$(b)$\Leftrightarrow$(e) of Theorem \ref{TBergman}.
We then extend these equivalences to $\alpha\in\mathbb R$ and prove
Theorem \ref{TBesov}.

We begin with the operators $D^t_s$.
The following lemma allows us to push the operator $D^t_s$ into integrals
involving reproducing kernels.
It can be proved in the same way as \cite[Lemma 3.3]{U3} using the
uniform convergence of the series expansion of $\mathcal R_\alpha$.
We omit the details.

\begin{lemma}\label{LPushDst}
Let $\alpha>-1$, $s,t\in\mathbb R$ and $f\in L^1_\alpha(\mathbb B)$.
Then
\begin{equation*}
D^t_s \int_{\mathbb B}\mathcal R_\alpha(x,y)f(y)\,d\nu_\alpha(y)
=\int_{\mathbb B}D^t_s \mathcal R_\alpha(x,y)f(y)\,d\nu_\alpha(y).
\end{equation*}
\end{lemma}

\begin{proof}[Proof of Theorem \ref{TBergman} (a)$\Leftrightarrow$(e)]
Suppose $f\in\mathcal B^p_\alpha$ and $s,t\in\mathbb R$ such that
\begin{equation}\label{alphaptlarge}
\alpha+pt>-1.
\end{equation}
We show that $D^t_s\in\mathcal B^p_{\alpha+pt}$ and
there exists a constant $C=C(n,\alpha,p,s,t)$ such that
\begin{equation}\label{Dstfless}
\|D^t_s f\|_{\mathcal B^p_{\alpha+pt}}
\le C \|f\|_{\mathcal B^p_\alpha}.
\end{equation}

Pick $\beta$ such that
\begin{equation}\label{pickbeta}
\alpha+1<p(\beta+1).
\end{equation}
Then $f\in L^1_\beta$ by H\"{o}lder's inequality,
\begin{equation}\label{Holderin}
\|f\|_{L^1_\beta}
\lesssim \Bigl(\int_{\mathbb B}\vert f(x)\vert^p(1-\vert x\vert^2)^\alpha
d\nu(x)\Bigr)^{\frac{1}{p}}
\Bigl(\int_{\mathbb B} (1-\vert x\vert^2)^{(\beta-\frac{\alpha}{p})p'}
d\nu(x)\Bigr)^{\frac{1}{p'}}
\lesssim \|f\|_{\mathcal B^p_\alpha}
\end{equation}
since $(\beta-\alpha/p)p'>-1$ by \eqref{pickbeta},
and the reproducing formula
\begin{equation}\label{reproducing}
f(x)=\int_{\mathbb B} \mathcal R_\beta(x,y)f(y)\,d\nu_\beta(y)
\end{equation}
holds by Theorem \ref{TProjBergman}.
Applying $D^t_s$ to both sides and pushing it into the integral
with Lemma \ref{LPushDst} shows
\begin{equation*}
\vert D^t_s f(x)\vert \le\int_{\mathbb B}
\vert D^t_s \mathcal R_\beta(x,y)\vert\, \vert f(y)\vert\,d\nu_\beta(y)
\end{equation*}
and estimating $\vert D^t_s \mathcal R_\beta\vert$ with Theorem \ref{TDstKernel}
we obtain (note that $\beta+t>-1$ by \eqref{alphaptlarge} and \eqref{pickbeta})
\begin{equation*}
(1-\vert x\vert^2)^t\vert D^t_s f(x)\vert
\lesssim (1-\vert x\vert^2)^t\int_{\mathbb B}
\frac{\vert f(y)\vert}{[x,y]^{n+\beta+t}}\,d\nu_\beta(y).
\end{equation*}
The right-hand side is $E_{\beta,t}(\vert f\vert)$ and because
$E_{\beta,t}\colon L^p_\alpha\to L^p_\alpha$ is bounded by Lemma \ref{LEbdd},
we deduce
$\|(1-\vert x\vert^2)^t D^t_s f\|_{L^p_\alpha}\lesssim \|f\|_{\mathcal B^p_\alpha}$
and \eqref{Dstfless} follows.

To see (e)$\Rightarrow$(a), suppose there exist $s,t$ with $\alpha+pt>-1$
such that $D^t_s f\in\mathcal B^p_{\alpha+pt}$.
Then, by the previous part, $D^{-t}_{s+t} D^t_s f\in\mathcal B^p_\alpha$
and because $D^{-t}_{s+t} D^t_s f=f$ by Lemma \ref{LPropDst} (iv), we conclude that
$f\in\mathcal B^p_\alpha$ and
$\|f\|_{\mathcal B^p_\alpha}\lesssim \|D^t_s f\|_{\mathcal B^p_{\alpha+pt}}$.
\end{proof}

We next consider tangential derivatives and prove part (a)$\Leftrightarrow$(b)
of Theorem \ref{TBergman}.

\begin{proof}[Proof of Theorem \ref{TBergman} (a)$\Rightarrow$(b)]
Suppose $f\in\mathcal B^p_\alpha$ and $k\ge 1$.
We show that $T^k f\in\mathcal B^p_{\alpha+pk}$ for all $T^k\in\mathcal T^k$,
and there exists a constant $C=C(n,\alpha,p,k)$ such that
\begin{equation*}
\vert f(0)\vert+
\sum_{T^k\in\mathcal T^k}\|T^k f\|_{\mathcal B^p_{\alpha+pk}}
\le C \|f\|_{\mathcal B^p_\alpha}.
\end{equation*}

The proof is similar to the proof of part (a)$\Rightarrow$(e).
We first consider first order tangential derivatives.
Pick $\beta$ with $\alpha+1<p(\beta+1)$.
Then, the reproducing formula \eqref{reproducing} holds and applying
$T_{i,j}$ to both sides and using
$\vert T_{i,j} g(x)\vert \le \vert\nabla g(x)\vert$ shows
\begin{equation*}
\vert T_{i,j} f(x)\vert \le\int_{\mathbb B}
\vert \nabla_x \mathcal R_\beta(x,y)\vert\, \vert f(y)\vert\,d\nu_\beta(y),
\end{equation*}
where $\nabla_x$ means $\nabla$ is taken with respect to $x$.
Estimating $\vert\nabla_x \mathcal R_\beta\vert$ with
Theorem \ref{TPartialKernel}(i) we obtain
\begin{equation*}
(1-\vert x\vert^2)\vert T_{i,j} f(x)\vert
\lesssim (1-\vert x\vert^2)\int_{\mathbb B}
\frac{\vert f(y)\vert}{[x,y]^{n+\beta+1}}\,d\nu_\beta(y).
\end{equation*}
The right-hand side is $E_{\beta,1}(\vert f\vert)$ and Lemma \ref{LEbdd}
implies that $T_{i,j}f\in\mathcal B^p_{\alpha+p}$ with
$\|T_{i,j}f\|_{\mathcal B^p_{\alpha+p}}\lesssim \|f\|_{\mathcal B^p_\alpha}$.
Because $T_{i,j}f$ is $\mathcal H$-harmonic we can repeat this argument
and obtain
$\|T^k f\|_{\mathcal B^p_{\alpha+kp}}\lesssim \|f\|_{\mathcal B^p_\alpha}$
for all $T^k\in\mathcal T^k$.
Finally, $\vert f(0)\vert\lesssim \|f\|_{\mathcal B^p_\alpha}$ follows from
the boundedness of point evaluation functionals.
\end{proof}

Before showing part (b)$\Rightarrow$(a) of Theorem \ref{TBergman},
we prove a lemma.

\begin{lemma}\label{Ldividex}
Let $\alpha>-1$ and $1\le p<\infty$.
There exists $C=C(n,\alpha,p)$ such that for all $f\in\mathcal B^p_\alpha$,
\begin{equation*}
\Bigl\|\frac{f(x)-f(0)}{\vert x\vert}\Bigr\|_{L^p_\alpha}
\le C\|f\|_{\mathcal B^p_\alpha}.
\end{equation*}
\end{lemma}

\begin{proof}
Pick $\beta$ with $\alpha+1<p(\beta+1)$.
By the reproducing formula \eqref{reproducing},
\begin{equation*}
\vert f(x)-f(0)\vert\le \int_{\mathbb B}
\vert \mathcal R_\beta(x,y)-\mathcal R_\beta(0,y)\vert
\vert f(y)\vert\,d\nu_\beta(y).
\end{equation*}
Since $\mathcal R_\beta(x,y)-\mathcal R_\beta(0,y)
=\int_0^1\langle x,\nabla_x\mathcal R_\beta(tx,y)\rangle\, dt$,
applying first Theorem \ref{TPartialKernel}(i) and then Lemma \ref{LInt01}
shows
\begin{equation*}
\vert \mathcal R_\beta(x,y)-\mathcal R_\beta(0,y)\vert
\lesssim \vert x\vert \int_0^1\frac{1}{[tx,y]^{n+\beta+1}}\, dt
\lesssim\frac{\vert x\vert}{[x,y]^{n+\beta}}.
\end{equation*}
Thus,
\begin{equation}\label{foverxless}
\frac{\vert f(x)-f(0)\vert}{\vert x\vert}
\lesssim \int_{\mathbb B} \frac{\vert f(y)\vert}{[x,y]^{n+\beta}}\,d\nu_\beta(y),
\end{equation}
and since the right-hand side is $E_{\beta,0}(\vert f\vert)$, the result
follows from Lemma \ref{LEbdd}.
\end{proof}

\begin{proof}[Proof of Theorem \ref{TBergman} (b)$\Rightarrow$(a)]
We first consider the case $k=1$.
Suppose $f\in\mathcal H(\mathbb B)$ and
$T_{i,j}f\in\mathcal B^p_{\alpha+p}$ for all $1\le i<j\le n$.
We show that $f\in\mathcal B^p_\alpha$ and there exists $C=C(n,\alpha,p)$
such that
\begin{equation*}
\|f\|_{\mathcal B^p_\alpha}
\le C
\Bigl(\sum_{1\le i<j\le n} \|T_{i,j}f\|_{\mathcal B^p_{\alpha+p}}+\vert f(0)\vert\Bigr).
\end{equation*}

We first note that
$\|f\|_{\mathcal B^p_\alpha}
\sim\|\,\vert\nabla f\vert\,\|_{L^p_{\alpha+p}}+\vert f(0)\vert$
by \cite[Theorem 10.3.3]{St1}.
Thus, it suffices to show that
$\|\,\vert\nabla f\vert\,\|_{L^p_{\alpha+p}}
\lesssim\sum_{i<j} \|T_{i,j}f\|_{\mathcal B^p_{\alpha+p}}$.
Now, by the equality
$\vert x\vert^2\vert\nabla f(x)\vert^2
=\vert Nf(x)\vert^2+\sum_{i<j}\vert T_{i,j} f(x)\vert^2$
(see \cite[p.~56]{St1}), we have
\begin{equation*}
\vert \nabla f(x)\vert\le \frac{\vert Nf(x)\vert}{\vert x\vert}
+\sum_{i<j}\frac{\vert T_{i,j}f(x)\vert}{\vert x\vert}
\end{equation*}
and by Lemma \ref{Ldividex},
\begin{equation*}
\Bigl\|\frac{T_{i,j} f(x)}{\vert x\vert}\Bigr\|_{L^p_{\alpha+p}}
\lesssim \|T_{i,j}f\|_{\mathcal B^p_{\alpha+p}}.
\end{equation*}
Hence, all we need to show is
\begin{equation}\label{Noverx}
\Bigl\|\frac{N f(x)}{\vert x\vert}\Bigr\|_{L^p_{\alpha+p}}
\lesssim \sum_{i<j}\|T_{i,j}f\|_{\mathcal B^p_{\alpha+p}}.
\end{equation}

To see \eqref{Noverx}, let
\begin{equation*}
\Delta_\sigma=\sum_{i<j} T_{i,j}^2
\end{equation*}
be the spherical Laplacian.
Then, by previously proved part (a)$\Rightarrow$(b) of Theorem \ref{TBergman},
we have $\Delta_\sigma f\in\mathcal B^p_{\alpha+2p}$ and
\begin{equation}\label{Deltaf}
\|\Delta_\sigma f\|_{\mathcal B^p_{\alpha+2p}}
\lesssim \sum_{i<j} \|T_{i,j}f\|_{\mathcal B^p_{\alpha+p}}.
\end{equation}

For $x\in\mathbb B$, we write $x=r\zeta$ with $r=\vert x\vert$ and
$\zeta\in\mathbb S$.
Since
\begin{equation*}
r^2\Delta_h=(1-r^2)\big[(1-r^2)N^2+(n-2)(1+r^2)N+(1-r^2)\Delta_\sigma\bigr]
\end{equation*}
(see \cite[p.~28]{St1}) and $\Delta_h f=0$, we have
\begin{equation*}
(1-r^2)N^2f+(n-2)(1+r^2)Nf=-(1-r^2)\Delta_\sigma f.
\end{equation*}
To solve $Nf$, for fixed $\zeta\in\mathbb S$, write $g(r)=Nf(r\zeta)$.
Since $rg'(r)=N^2f(r\zeta)$, the above equation turns into
\begin{equation*}
g'(r)+(n-2)\frac{1+r^2}{r(1-r^2)}g(r)=-\frac{1}{r}\Delta_\sigma f(r\zeta)
\end{equation*}
and solving this first order linear equation shows
\begin{equation*}
Nf(r\zeta)=-r^{2-n}(1-r^2)^{n-2}
\int_0^r t^{n-3}(1-t^2)^{2-n}\Delta_\sigma f(t\zeta)\, dt,
\end{equation*}
and using $t\le r$ and $1-r^2\le 1-t^2$, we deduce
\begin{equation}\label{Nfoverxless}
\vert Nf(r\zeta)\vert\le\int_0^r\frac{\vert\Delta_\sigma f(t\zeta)\vert}{t}\,dt
=\int_0^1\frac{\vert\Delta_\sigma f(t r\zeta)\vert}{t}\,dt.
\end{equation}

Pick $\beta$ satisfying $\alpha+2p+1<p(\beta+1)$.
Since $\Delta_\sigma f\in\mathcal B^p_{\alpha+2p}$ and $\Delta_\sigma f(0)=0$,
by the inequality \eqref{foverxless},
\begin{equation*}
\frac{\vert\Delta_\sigma f(x)\vert}{\vert x\vert}
\lesssim \int_{\mathbb B}
\frac{\vert\Delta_\sigma f(y)\vert}{[x,y]^{n+\beta}}\,d\nu_\beta(y)
\end{equation*}
and so, by \eqref{Nfoverxless}, Fubini's theorem and Lemma \ref{LInt01},
\begin{align*}
\frac{\vert Nf(x)\vert}{\vert x\vert}
\le \int_0^1\frac{\vert\Delta_\sigma f(tx)\vert}{t\vert x\vert}dt
&\lesssim \int_{\mathbb B}\vert\Delta_\sigma f(y)\vert
\int_0^1\frac{dt}{[tx,y]^{n+\beta}}\,d\nu_\beta(y)\\
&\lesssim \int_{\mathbb B}
\frac{\vert\Delta_\sigma f(y)\vert}{[x,y]^{n+\beta-1}}\, d\nu_\beta(y).
\end{align*}
The right-hand side is $(1-\vert x\vert^2)E_{\beta,-1}(\vert\Delta_\sigma f\vert)$
and it follows from Lemma \ref{LEbdd} that
$(1-\vert x\vert^2)^{-1}\frac{Nf(x)}{\vert x\vert}\in L^p_{\alpha+2p}$,
that is
$\frac{Nf(x)}{\vert x\vert}\in L^p_{\alpha+p}$ and
\begin{equation*}
\Bigl\|\frac{N f(x)}{\vert x\vert}\Bigr\|_{L^p_{\alpha+p}}
\lesssim \|\Delta_\sigma f\|_{\mathcal B^p_{\alpha+2p}}
\lesssim \sum_{i<j}\|T_{i,j}f\|_{\mathcal B^p_{\alpha+p}},
\end{equation*}
where in the last step we use \eqref{Deltaf}.
Thus, \eqref{Noverx} holds and the case $k=1$ is proved.

For the general case, let $k\ge 1$ and suppose $T^k f\in\mathcal B^p_{\alpha+pk}$
for all $T^k\in\mathcal T^k$.
Then, by what we have shown above, $T^{k-1}f\in\mathcal B^p_{\alpha+p(k-1)}$ and
\begin{equation*}
\|T^{k-1} f\|_{\mathcal B^p_{\alpha+p(k-1)}}
\lesssim \sum_{T^k\in\mathcal T^k}\|T^k f\|_{\mathcal B^p_{\alpha+pk}}
\end{equation*}
for all $T^{k-1}\in\mathcal T^{k-1}$.
Repeating this argument we conclude that $f\in\mathcal B^p_\alpha$ and
\begin{equation*}
\|f\|_{\mathcal B^p_\alpha}
\lesssim \sum_{T^k\in\mathcal T^k}\|T^k f\|_{\mathcal B^p_{\alpha+pk}}+\vert f(0)\vert.
\qedhere
\end{equation*}
\end{proof}

We next consider the case $\alpha\in\mathbb R$ and prove Theorem \ref{TBesov}.
We first show that the operators $D^t_s$ and the tangential derivatives
commute.

\begin{lemma}\label{LDTcommute}
$D^t_s(T_{i,j}f)=T_{i,j}(D^t_s f)$ for all $s,t\in\mathbb R$, $1\le i<j\le n$
and $f\in\mathcal H(\mathbb B)$.
\end{lemma}

\begin{proof}
Suppose $f$ has homogeneous expansion $f(x)=\sum_{m=0}^\infty S_m(\vert x\vert)f_m(x)$,
where $f_m\in H_m$.
Since $S_m(\vert x\vert)$ is radial, $T_{i,j} S_m(\vert x\vert)=0$ and therefore
\begin{equation}\label{tijf}
T_{i,j} f(x)=\sum_{m=0}^\infty S_m(\vert x\vert) T_{i,j}f_m(x).
\end{equation}
Because $T_{i,j} f_m$ is harmonic and homogeneous of degree $m$, $T_{i,j}f_m\in H_m$
and \eqref{tijf} is the (unique) homogeneous expansion of the $\mathcal H$-harmonic
function $T_{i,j}f$.
Thus,
\begin{equation*}
D^t_sT_{i,j} f(x)=\sum_{m=0}^\infty\frac{c_m(s+t)}{c_m(s)}S_m(\vert x\vert) T_{i,j}f_m(x)
\end{equation*}
which also equals $T_{i,j}D^t_s f$ by the same reasoning.
\end{proof}

\begin{proof}[Proof of Theorem \ref{TBesov}]
(a)$\Rightarrow$(b):
The case $k=0$ follows from Theorem \ref{TBergman}, so we assume $k\ge 1$.
Suppose there exists $k\ge 1$ with $\alpha+pk>-1$ such that
$T^k f\in\mathcal B^p_{\alpha+pk}$ for every $T^k\in\mathcal T^k$.
Let $s,t\in\mathbb R$ with $\alpha+pt>-1$.
We show that $D^t_s f\in\mathcal B^p_{\alpha+pt}$ and there exists
$C=C(n,\alpha,p,k,s,t)$ such that
\begin{equation*}
\|D^t_s f\|_{\mathcal B^p_{\alpha+pt}}
\le C \bigl(\sum_{T^k\in\mathcal T^k}\|T^k f\|_{\mathcal B^p_{\alpha+pk}}
+\vert f(0)\vert\bigr).
\end{equation*}

We first note that we have $\alpha+pk>-1$, $\alpha+pt>-1$ and $\alpha+pt+pk>-1$
and throughout the proof we will be in the Bergman zone.
Now, by part (a)$\Rightarrow$(e) of Theorem \ref{TBergman},
$T^k f\in\mathcal B^p_{\alpha+pk}$ implies
$D^t_s(T^k f)\in\mathcal B^p_{\alpha+pk+pt}$ and
\begin{equation}\label{DtTk}
\|D^t_s(T^k f)\|_{\mathcal B^p_{\alpha+pk+pt}}
\lesssim \|T^k f\|_{\mathcal B^p_{\alpha+pk}}.
\end{equation}
Because $D^t_s(T^kf)=T^k(D^t_s f)$ by Lemma \ref{LDTcommute},
we deduce that $T^k(D^t_s f)\in\mathcal B^p_{\alpha+pt+pk}$.
Since this is true for all $T^k\in\mathcal T^k$, by part
(b)$\Rightarrow$(a) of Theorem \ref{TBergman}, we see that
$D^t_s f\in\mathcal B^p_{\alpha+pt}$ and
\begin{equation*}
\|D^t_s f\|_{\mathcal B^p_{\alpha+pt}}
\lesssim \sum_{T^k\in\mathcal T^k}\|T^k (D^t_s f)\|_{\mathcal B^p_{\alpha+pt+pk}}
+\vert D^t_s f(0)\vert
\lesssim \sum_{T^k\in\mathcal T^k}\|T^k f\|_{\mathcal B^p_{\alpha+pk}}
+\vert f(0)\vert
\end{equation*}
by \eqref{DtTk} and the fact that $D^t_s f(0)=f(0)$.

(b)$\Rightarrow$(a):
Suppose there exist $s,t\in\mathbb R$ with $\alpha+pt>-1$
such that $D^t_s f\in\mathcal B^p_{\alpha+pt}$.
We show that for all $k\ge 0$ satisfying $\alpha+pk>-1$,
we have $T^k f\in\mathcal B^p_{\alpha+pk}$ for all $T^k\in\mathcal T^k$
and there exists $C=C(n,\alpha,p,s,t,k)$ such that
\begin{equation*}
\sum_{T^k\in\mathcal T^k}\|T^k f\|_{\mathcal B^p_{\alpha+pk}}
+\vert f(0)\vert
\le C \|D^t_s f\|_{\mathcal B^p_{\alpha+pt}}.
\end{equation*}

The proof is similar to the previous part.
By part (a)$\Rightarrow$(b) of Theorem \ref{TBergman},
$D^t_s f\in\mathcal B^p_{\alpha+pt}$ implies
$T^k(D^t_s f)\in\mathcal B^p_{\alpha+pt+pk}$
and
\begin{equation}\label{tanless}
\|T^k(D^t_s f)\|_{\mathcal B^p_{\alpha+pt+pk}}
\lesssim \|D^t_s f\|_{\mathcal B^p_{\alpha+pt}}.
\end{equation}
Hence, by Lemma \ref{LDTcommute},
$D^t_s(T^k f)\in\mathcal B^p_{\alpha+pk+pt}$
and as we stay in the Bergman zone,
part (e)$\Rightarrow$(a) of Theorem \ref{TBergman} implies
$T^k f\in\mathcal B^p_{\alpha+pk}$ with
\begin{equation*}
\|T^k f\|_{\mathcal B^p_{\alpha+pk}}
\lesssim \|D^t_s(T^k f)\|_{\mathcal B^p_{\alpha+pk+pt}}
\lesssim \|D^t_s f\|_{\mathcal B^p_{\alpha+pt}}
\end{equation*}
by \eqref{tanless}.
Finally, we have $\vert f(0)\vert=\vert D^t_s f(0)\vert
\lesssim \|D^t_s f\|_{\mathcal B^p_{\alpha+pt}}$
since point evaluation functionals are bounded on Bergman spaces.
\end{proof}

\begin{remark}\label{Rp2}
When $p=2$, the characterization of the Hilbert space $\mathcal B^2_\alpha$
given in \eqref{B2alpha} for $\alpha>-1$ holds also for $\alpha\le -1$.
To see this, pick $s,t\in\mathbb R$ with $\alpha+2t>-1$, and let
$f=\sum_{m=0}^\infty S_m(\vert x\vert)f_m\in\mathcal H(\mathbb B)$.
Then, $f\in\mathcal B^2_\alpha$ if and only if
$D^t_s f\in\mathcal B^2_{\alpha+2t}$ and this holds if and only if
\begin{equation*}
\sum_{m=1}^\infty\frac{1}{m^{\alpha+2t+1}}\frac{c_m^2(s+t)}{c_m^2(s)}
\|f_m\|^2_{L^2(\mathbb S)}<\infty.
\end{equation*}
Because $c_m(s+t)/c_m(s)\sim m^t$ by \eqref{cmasym}, the above condition
is equivalent to the condition given in \eqref{B2alpha}.

A straightforward calculation shows that for $\alpha\leq -1$ and
$f,g\in\mathcal B^2_\alpha$ with $f=\sum_{m=0}^\infty S_m(\vert x\vert)f_m$,
$g=\sum_{m=0}^\infty S_m(\vert x\vert)g_m$,
$\langle f,g\rangle_{\alpha}
=\sum_{m=0}^\infty\frac{1}{c_m(\alpha)}\langle f_m,g_m\rangle_{L^2(\mathbb S)}$
is an inner product on  $\mathcal B^2_\alpha$ with corresponding
reproducing kernel $\mathcal R_{\alpha}(x,y)$.
\end{remark}

\section{$\mathcal H$-Harmonic Bergman-Besov spaces}

In this section, we prove Theorems \ref{TIso}--\ref{TInclusion}.

\begin{proof}[Proof of Theorem \ref{TIso}]
Let $s\in\mathbb R$ and $t=(\alpha_2-\alpha_1)/p$.
We show that the operator
$D^t_s\colon\mathcal B^p_{\alpha_1}\to\mathcal B^p_{\alpha_2}$
is an isomorphism.
Pick large enough $\tau\in\mathbb R$ so that
\begin{equation*}
\alpha_1+p\tau>-1
\quad\text{and}\quad
\alpha_2+p\tau>-1.
\end{equation*}
Let $f\in\mathcal B^p_{\alpha_1}$.
Then, by Definition \ref{DBesov}, $D^\tau_s f\in\mathcal B^p_{\alpha_1+p\tau}$
and $\|D^\tau_s f\|_{\mathcal B^p_{\alpha_1+p\tau}}\sim\|f\|_{\mathcal B^p_{\alpha_1}}$.
Next, because $\alpha_1+p\tau+pt=\alpha_2+p\tau$,
by Theorem \ref{TBergman}, we have $D^t_s(D^\tau_s f)\in\mathcal B^p_{\alpha_2+p\tau}$
and $\|D^t_s(D^\tau_s f)\|_{\mathcal B^p_{\alpha_2+p\tau}}
\sim \|D^\tau_s f\|_{\mathcal B^p_{\alpha_1+p\tau}}$.
Since the operators $D^t_s$ commute, we deduce
$D^\tau_s(D^t_s f)\in\mathcal B^p_{\alpha_2+p\tau}$
and this implies that $D^t_s f\in\mathcal B^p_{\alpha_2}$ with
$\|D^t_s f\|_{\mathcal B^p_{\alpha_2}}
\sim \|D^\tau_s(D^t_s f)\|_{\mathcal B^p_{\alpha_2+p\tau}}
\sim \|f\|_{\mathcal B^p_{\alpha_1}}$.
Finally, the inverse of $D^t_s$ is $D^{-t}_{s+t}$ and by a similar argument,
if $f\in\mathcal B^p_{\alpha_2}$, then $D^{-t}_{s+t}f\in\mathcal B^p_{\alpha_1}$
and $\|D^{-t}_{s+t}f\|_{\mathcal B^p_{\alpha_1}}
\sim \|f\|_{\mathcal B^p_{\alpha_2}}$.
\end{proof}

We will use the following lemma in the proof of Theorem \ref{TProjBesov}.

\begin{lemma}\label{LRbig}
Let $\alpha\in\mathbb R$.
There exists $\varepsilon=\varepsilon(n,\alpha)>0$ such that
$\mathcal R_\alpha(x,y)\ge \frac{1}{2}$ for all
$\vert x\vert<\varepsilon$ and $y\in\mathbb B$.
\end{lemma}

\begin{proof}
We have
\begin{equation*}
\mathcal R_\alpha(x,y)=1+\sum_{m=1}^\infty
c_m(\alpha)S_m(\vert x\vert)S_m(\vert y\vert)Z_m(x,y)
\end{equation*}
and the lemma follows from the estimates \eqref{cmasym},
$S_m(r)\lesssim m^{n/2-1}$ \cite[Lemma 2.13]{U}, and
$\vert Z_m(x,y)\vert\lesssim m^{n-2}\vert x\vert^m\vert y\vert^m$.
\end{proof}

\begin{proof}[Proof of Theorem \ref{TProjBesov}]
Suppose $\alpha+1<p(\beta+1)$.
Pick $s,t\in\mathbb R$ with $\alpha+pt>-1$.
Note that these two inequalities imply $\beta+t>-1$.

Let $\varphi\in L^p_\alpha$.
Then $\varphi\in L^1_\beta$ as in \eqref{Holderin}.
To verify $P_\beta\varphi\in\mathcal B^p_\alpha$, we show that
$D^t_s(P_\beta\varphi)\in\mathcal B^p_{\alpha+pt}$.
Now, by Lemma \ref{LPushDst},
\begin{equation*}
D^t_s(P_\beta\varphi)(x)
=\int_{\mathbb B}D^t_s\mathcal R_\beta(x,y)\varphi(y)\,d\nu_\beta(y),
\end{equation*}
and estimating $\vert D^t_s\mathcal R_\beta\vert$ with Theorem \ref{TDstKernel}
shows
\begin{equation*}
(1-\vert x\vert^2)^t\vert D^t_s(P_\beta\varphi)(x)\vert
\lesssim (1-\vert x\vert^2)^t
\int_{\mathbb B}\frac{\vert\varphi(y)\vert}{[x,y]^{n+\beta+t}}d\nu_\beta(y).
\end{equation*}
The right-hand side is $E_{\beta,t}(\vert\varphi\vert)$ and Lemma \ref{LEbdd}
implies $(1-\vert x\vert^2)^t D^t_s(P_\beta \varphi)\in L^p_\alpha$, that is
$D^t_s(P_\beta \varphi)\in\mathcal B^p_{\alpha+pt}$, and
$\|P_\beta \varphi\|_{\mathcal B^p_\alpha}
\sim \|D^t_s(P_\beta \varphi)\|_{\mathcal B^p_{\alpha+pt}}
\lesssim \|\varphi\|_{L^p_\alpha}$.
Thus, $P_\beta \colon L^p_\alpha\to\mathcal B^p_\alpha$ is bounded.

We next show that $P_\beta$ is right-invertible.
Let $f\in\mathcal B^p_\alpha$.
Then, $D^t_\beta f\in\mathcal B^p_{\alpha+pt}$.
In addition, applying H\"older's inequality similar to \eqref{Holderin}
shows $D^t_\beta f\in L^1_{\beta+t}$ and
\begin{equation}\label{DtinL1}
\|D^t_\beta f\|_{L^1_{\beta+t}}\lesssim \|D^t_\beta f\|_{\mathcal B^p_{\alpha+pt}}
\sim\|f\|_{\mathcal B^p_\alpha}.
\end{equation}
Since $\alpha+pt>-1$, $\beta+t>-1$ and $\alpha+pt+1<p(\beta+t+1)$,
by Theorem \ref{TProjBergman},
\begin{equation*}
D^t_\beta f(x)
=\int_{\mathbb B}\mathcal R_{\beta+t}(x,y) D^t_\beta f(y)\,d\nu_{\beta+t}(y).
\end{equation*}
We apply $D^{-t}_{\beta+t}$ to both sides and push it into the integral
using Lemma \ref{LPushDst}.
By parts (ii) and (iv) of Lemma \ref{LPropDst}, we obtain
\begin{equation*}
f(x)=\frac{V_\beta}{V_{\beta+t}}\int_{\mathbb B}\mathcal R_\beta(x,y)
\bigl[(1-\vert y\vert^2)^t D^t_\beta f(y)\bigr]\,d\nu_\beta(y).
\end{equation*}
Thus $f=P_\beta\varphi$, where
$\varphi=\frac{V_\beta}{V_{\beta+t}}(1-\vert x\vert^2)^t D^t_\beta f$,
and
$\|\varphi\|_{L^p_\alpha}\sim\|D^t_\beta f\|_{\mathcal B^p_{\alpha+pt}}
\sim \|f\|_{\mathcal B^p_\alpha}$.

Suppose now that $P_\beta\colon L^p_\alpha\to\mathcal B^p_\alpha$ is bounded.
We show that the inequality \eqref{projineq} holds.
Assume first that $p>1$.
Let $\varphi(x)=(1-\vert x\vert^2)^{-(\alpha+1)/p}
\bigl(1+\log\frac{1}{1-\vert x\vert^2}\bigr)^{-1}$.
Then $\varphi\in L^p_\alpha$ and with $\varepsilon$ as given in Lemma \ref{LRbig}
for $\mathcal R_\beta$,
for $\vert x\vert<\varepsilon$, we have
\begin{equation*}
P_\beta\varphi(x)\ge \frac{1}{2V_\beta}\int_{\mathbb B}
\frac{(1-\vert y\vert^2)^{-(\alpha+1)/p+\beta}}
{1+\log\frac{1}{1-\vert y\vert^2}}\,d\nu(y).
\end{equation*}
Because $P_\beta\varphi\in\mathcal B^p_\alpha$, the above integral
must be finite, that is we must have $-(\alpha+1)/p+\beta>-1$ which
is same as \eqref{projineq}.

We next consider the case $p=1$.
For $\delta>0$, let $\varphi(x)=(1-\vert x\vert^2)^{-(\alpha+1)+\delta}$.
Then $\varphi\in L^1_\alpha$ and since $P_\beta\varphi\in\mathcal B^p_\alpha$,
similar to above, by Lemma \ref{LRbig}, the integral
$\int_{\mathbb B}(1-\vert y\vert^2)^{-(\alpha+1)+\delta+\beta}\,d\nu(y)$
must be finite for all $\delta>0$.
This shows $\alpha\le\beta$.
To see that we must have $\alpha<\beta$, assume $\alpha=\beta$.
We will obtain a contradiction with Theorem \ref{TProjBergman}.
For this, pick $t\in\mathbb R$ with $\alpha+t>-1$.
Note that the multiplication operator $M_t\colon L^1_{\alpha+t}\to L^1_\alpha$,
$M_t\varphi(x)=(1-\vert x\vert^2)^t\varphi(x)$ is an isometry.
Also, by Theorem \ref{TIso},
$D^t_\beta\colon \mathcal B^1_\alpha\to\mathcal B^1_{\alpha+t}$
is an isomorphism.
Therefore, the composition
$D^t_\beta\circ P_\beta\circ M_t\colon
L^1_{\alpha+t}\to\mathcal B^1_{\alpha+t}$
is bounded.
Now, for $\varphi\in L^1_{\alpha+t}$,
\begin{equation*}
D^t_\beta\circ P_\beta\circ M_t\varphi(x)
=D^t_\beta\int_{\mathbb B}
\mathcal R_\beta(x,y)(1-\vert y\vert^2)^t\varphi(y)\,d\nu_\beta(y)
\end{equation*}
and we can push $D^t_\beta$ into the integral with Lemma \ref{LPushDst}
since $\varphi\in L^1_{\beta+t}$ by the assumption $\alpha=\beta$.
Because $D^t_\beta\mathcal R_\beta=\mathcal R_{\beta+t}$ by part (ii) of
Lemma \ref{LPropDst}, we conclude that
$D^t_\beta\circ P_\beta\circ M_t=C P_{\beta+t}$.
Thus, $P_{\beta+t}\colon L^p_{\alpha+t}\to\mathcal B^p_{\alpha+t}$ is
bounded and this contradicts Theorem \ref{TProjBergman}.
\end{proof}

Corollary \ref{CPointeva} follows from \eqref{IntRep}.
For fixed $x_0\in\mathbb B$, we have
$\vert\mathcal R_\beta(x_0,y)\vert\lesssim 1$ by Theorem \ref{TDstKernel}
(with $t=0$) since $[x_0,y]\ge 1-\vert x_0\vert$.
Thus, together with \eqref{DtinL1}, we have
$\vert f(x_0)\vert\lesssim \|D^t_\beta f\|_{L^1_{\beta+t}}
\lesssim\|f\|_{\mathcal B^p_\alpha}$.

\begin{proof}[Proof of Theorem \ref{TDuality}]
The proof follows from the Bergman case and the isomorphisms in
Theorem \ref{TIso}.
First consider the case $1<p<\infty$.
If $f\in\mathcal B^p_\alpha$, then $D^t_s f\in\mathcal B^p_{\alpha+pt}$
and we use the norm
$\|f\|_{\mathcal B^p_\alpha}=\|D^t_s f\|_{\mathcal B^p_{\alpha+pt}}$.
Similarly, if $g\in\mathcal B^{p'}_\alpha$, then
$D^{t'}_s g\in\mathcal B^{p'}_{\alpha+p't'}=\mathcal B^{p'}_{\alpha+pt}$
and
$\|g\|_{\mathcal B^{p'}_\alpha}=\|D^{t'}_s g\|_{\mathcal B^{p'}_{\alpha+pt}}$.
For $g\in\mathcal B^{p'}_\alpha$,
define $\Lambda_g\colon \mathcal B^p_\alpha\to\mathbb C$,
\begin{equation*}
\Lambda_g(f)=\int_{\mathbb B}
D^t_s f(x) D^{t'}_s g(x)\,d\nu_{\alpha+pt}(x).
\end{equation*}
Then, by H\"older's inequality, $\Lambda_g\in(\mathcal B^p_\alpha)^*$ and
$\|\Lambda_g\|\le\|g\|_{\mathcal B^{p'}_\alpha}$.

Suppose now that $\Lambda\in(\mathcal B^p_\alpha)^*$.
Then, since $D^{-t}_{s+t}\colon\mathcal B^p_{\alpha+pt}\to\mathcal B^p_\alpha$
is an isomorphism, $\Lambda\circ D^{-t}_{s+t}\in(\mathcal B^p_{\alpha+pt})^*$.
By the duality in the Bergman case \cite[Corollary 1.4]{U}, there exists
$\tilde{g}\in\mathcal B^{p'}_{\alpha+pt}$ with
$\|\tilde{g}\|_{\mathcal B^{p'}_{\alpha+pt}}
\lesssim\|\Lambda\circ  D^{-t}_{s+t}\|
\lesssim \|\Lambda\|$ such that
\begin{equation}\label{dualeq}
\Lambda\circ D^{-t}_{s+t}(\tilde{f})
=\int_{\mathbb B}\tilde{f}(x)\tilde{g}(x)\,d\nu_{\alpha+pt}(x)
\qquad (\tilde{f}\in\mathcal B^p_{\alpha+pt}).
\end{equation}
Let $g=D^{-t'}_{s+t'}\tilde{g}$.
Then $\tilde{g}=D^{t'}_{s}g$ by Lemma \ref{LPropDst} part (iv).
Also, by Theorem \ref{TIso}, we have $g\in\mathcal B^{p'}_\alpha$ and
$\|g\|_{\mathcal B^{p'}_{\alpha}}
=\|\tilde{g}\|_{\mathcal B^{p'}_{\alpha+pt}}\lesssim\|\Lambda\|$.
Now, for $f\in\mathcal B^p_\alpha$,
$\tilde{f}=D^t_s f\in\mathcal B^p_{\alpha+pt}$
and inserting these into \eqref{dualeq} shows $\Lambda=\Lambda_g$.

To see that $g$ is unique, for fixed $x_0\in\mathbb B$, let
$f_{x_0}(x)=D^{-t}_{s+t}\mathcal R_{\alpha+pt}(x_0,x)$.
Then $D^t_s f_{x_0}=\mathcal R_{\alpha+pt}(x_0,x)$ is bounded
and so $D^t_s f_{x_0}\in\mathcal B^p_{\alpha+pt}$ and
$f_{x_0}\in\mathcal B^p_\alpha$.
We have
\begin{equation*}
\Lambda_g(f_{x_0})=\int_{\mathbb B}\mathcal R_{\alpha+pt}(x_0,x)
D^{t'}_sg(x)\,d\nu_{\alpha+pt}(x)
=D^{t'}_sg(x_0)
\end{equation*}
by the reproducing property and so if $g_1\neq g_2$, then
$D^{t'}_sg_1\neq D^{t'}_s g_2$ and $\Lambda_{g_1}\neq\Lambda_{g_2}$.

The case $p=1$ follows from the Bergman case
\cite[Theorem 1.3]{U3} in a similar way.
For $f\in\mathcal B^1_\alpha$, we use the norm
$\|f\|_{\mathcal B^1_\alpha}=\|D^t_s f\|_{\mathcal B^1_{\alpha+t}}$.
For $g\in\mathcal B$,
let $\Lambda_g\colon\mathcal B^1_\alpha\to\mathbb C$,
\begin{equation*}
\Lambda_g(f)
=\lim_{r\to 1^-}\int_{r\mathbb B} D^t_s f(x)g(x)\,d\nu_{\alpha+t}(x).
\end{equation*}
Then, by \cite[Eqn~6.1]{U3},
$\vert\Lambda_g(f)\vert
\lesssim \|D^t_s f\|_{\mathcal B^1_{\alpha+t}}\|g\|_{\mathcal B}$
and so $\|\Lambda_g\|\lesssim\|g\|_{\mathcal B}$.
Next, if $f\in(\mathcal B^1_\alpha)^*$, then
$\Lambda\circ D^{-t}_{s+t}\in(\mathcal B^1_{\alpha+t})^*$
and, again by the Bergman case, there exists a unique $g\in\mathcal B$
with $\|g\|_{\mathcal B}\lesssim \|\Lambda\|$ such that
\begin{equation*}
\Lambda\circ D^{-t}_{s+t}(\tilde{f})
=\lim_{r\to 1^-}\int_{r\mathbb B}\tilde{f}(x)g(x)\,d\nu_{\alpha+t}(x)
\qquad (\tilde{f}\in\mathcal B^1_{\alpha+t}).
\end{equation*}
For $f\in\mathcal B^1_\alpha$, taking $\tilde{f}=D^t_s f$ shows
$\Lambda=\Lambda_g$.
The statement for the predual of $\mathcal B^1_\alpha$ follows
from \cite[Theorem 1.2]{U3} in an analogous manner.
\end{proof}

\begin{proof}[Proof of Theorem \ref{TAtomic}]
Let $f\in\mathcal B^p_\alpha$.
Pick $t\in\mathbb R$ with $\alpha+pt>-1$.
Then $D^t_s f\in\mathcal B^p_{\alpha+pt}$ and by the atomic
decomposition of Bergman spaces \cite[Theorem 1.1 and Eqn.~(5)]{U2},
there exists $\{\lambda_m\}\in\ell^p$ with
$\|\{\lambda_m\}\|_{\ell^p}\sim \|D^t_s f\|_{\mathcal B^p_{\alpha+pt}}
\sim\|f\|_{\mathcal B^p_\alpha}$ such that
\begin{equation*}
D^t_s f(x)=\sum_{m=1}^\infty
\lambda_m(1-\vert a_m\vert^2)^{(s+n)-(\alpha+n)/p}
\mathcal R_{s+t}(x,a_m).
\end{equation*}
We apply $D^{-t}_{s+t}$ to both sides.
On the left $D^{-t}_{s+t}D^t_s f=f$, and on the right we can push
$D^{-t}_{s+t}$ into the series by Lemma \ref{LPropDst} (i) and uniform
convergence of the above series on compact subsets.
Since $D^{-t}_{s+t}\mathcal R_{s+t}=\mathcal R_s$, the desired result
follows.
\end{proof}

\begin{proof}[Proof of Theorem \ref{TInclusion}]
In the Bergman zone it is proved in \cite[Theorem 1.3]{U2} that
if $\alpha,\beta>-1$ and $1\le p\le q<\infty$, then
\begin{equation*}
\mathcal B^p_\alpha\subset\mathcal B^q_\beta\quad
  \text{if and only if}\quad
  \frac{\alpha+n}{p}\leq\frac{\beta+n}{q}.
\end{equation*}
For the general case, let $\alpha,\beta\in\mathbb R$.
Pick $s,t\in\mathbb R$ with $t$ large enough that $\alpha+pt>-1$
and $\beta+qt>-1$.
Then, by Theorem \ref{TIso},
$D^t_s(\mathcal B^p_\alpha)=\mathcal B^p_{\alpha+pt}$
and
$D^t_s(\mathcal B^q_\beta)=\mathcal B^q_{\beta+qt}$.
Thus,  $\mathcal B^p_\alpha\subset\mathcal B^q_\beta$
if and only if
$\mathcal B^p_{\alpha+pt}\subset\mathcal B^q_{\beta+qt}$
and this holds if and only if
$\frac{\alpha+pt+n}{p}\le\frac{\beta+qt+n}{q}$.
This last inequality is same as
$\frac{\alpha+n}{p}\le\frac{\beta+n}{q}$.
The case $q<p$ follows from the $q<p$ case of \cite[Theorem 1.3]{U2}
in a similar way.
\end{proof}

\section{Partial Derivative Characterization}

We first consider the case $\alpha>-1$ and prove part (a)$\Leftrightarrow$(c)
of Theorem \ref{TBergman}.
The proof of the implication (a)$\Rightarrow$(c) is based on the work
done before; reproducing formula of Theorem \ref{TProjBergman}, estimates
of the derivatives of the reproducing kernels in Theorem \ref{TPartialKernel}
and boundedness of the operator $E$ in Lemma \ref{LEbdd}.
The proof of part (c)$\Rightarrow$(a) uses a different approach and is based on
Hardy and Poincare inequalities.

\begin{proof}[Proof of Theorem \ref{TBergman} (a)$\Rightarrow$(c)]
Suppose $f\in\mathcal B^p_\alpha$ and $k\ge 1$.
We show that there exists a constant $C=C(n,\alpha,p,k)$ such that
\begin{align}
\textup{(i)}\ &\|\partial^\kappa f\|_{L^p_{\alpha+pk}}\le C \|f\|_{\mathcal B^p_\alpha}
\ \text{for every multi-index $\kappa$ with $\vert\kappa\vert=k$,}\label{partkf}\\
\textup{(ii)} &\sum_{\vert\kappa\vert\le k-1}\vert\partial^\kappa f(0)\vert
\le C \|f\|_{\mathcal B^p_\alpha}.\label{kappaless}
\end{align}

Let $\kappa$ be a multi-index with $\vert\kappa\vert=k$.
Pick $\beta$ with $\alpha+1<p(\beta+1)$.
By Theorem \ref{TProjBergman}, the reproducing formula
\begin{equation}\label{reprform}
f(x)=\int_{\mathbb B}\mathcal R_\beta(x,y)f(y)\,d\nu_\beta(y)
\end{equation}
holds, and so
\begin{equation}\label{partkfless}
\vert\partial^\kappa f(x)\vert
\lesssim \int_{\mathbb B}
\vert\partial^\kappa\mathcal R_\beta(x,y)\vert\,\vert f(y)\vert\,d\nu_\beta(y).
\end{equation}
We estimate $\vert\partial^\kappa \mathcal R_\beta\vert$ with
Theorem \ref{TPartialKernel} in three cases.

If $k\le n-2$, then part (i) of Theorem \ref{TPartialKernel} implies
(note that $\beta>-1$)
\begin{equation*}
(1-\vert x\vert^2)^k\vert\partial^\kappa f(x)\vert
\lesssim (1-\vert x\vert^2)^k \int_{\mathbb B}
\frac{\vert f(y)\vert}{[x,y]^{n+\beta+k}}\,d\nu_\beta(y)
=E_{\beta,k}(\vert f\vert),
\end{equation*}
and \eqref{partkf} follows from Lemma \ref{LEbdd}.

If $k=n-1$, then pick $0<\varepsilon<1$.
By part (b) of Theorem \ref{TPartialKernel}(ii),
\begin{equation*}
(1-\vert x\vert^2)^{n-1}\vert\partial^\kappa f(x)\vert
\lesssim (1-\vert x\vert^2)^{n-1-\varepsilon} \int_{\mathbb B}
\frac{\vert f(y)\vert}{[x,y]^{n+\beta+n-1-\varepsilon}}\,d\nu_\beta(y),
\end{equation*}
and \eqref{partkf} holds again by Lemma \ref{LEbdd}.

If $k\ge n$, then by part (iii) of Theorem \ref{TPartialKernel},
\begin{equation*}
(1-\vert x\vert^2)^k\vert\partial^\kappa f(x)\vert
\lesssim (1-\vert x\vert^2)^{n-1} \int_{\mathbb B}
\frac{\vert f(y)\vert}{[x,y]^{n+\beta+n-1}}\,d\nu_\beta(y),
\end{equation*}
and \eqref{partkf} follows from Lemma \ref{LEbdd}.

Finally, to see \eqref{kappaless}, note first that we have
$\|f\|_{L^1_\beta}\lesssim \|f\|_{\mathcal B^p_\alpha}$
by \eqref{Holderin}.
Now, putting $x=0$ after estimating $\vert\partial^\kappa\mathcal R_\beta\vert$
in \eqref{partkfless} and using $[0,y]=1$ gives \eqref{kappaless}.
\end{proof}

We next show part (c)$\Rightarrow$(a) of Theorem \ref{TBergman}.
Suppose that $f\in\mathcal H(\mathbb B)$ and there exists $k\ge 1$
such that $\partial^\kappa f\in L^p_{\alpha+pk}$ for every multi-index
$\vert\kappa\vert=k$.
We show that $f\in\mathcal B^p_\alpha$ and there exists a constant
$C=C(n,\alpha,p,k)$ such that
\begin{equation}\label{dimpa}
\|f\|_{\mathcal B^p_\alpha}
\le C\biggl(\sum_{\vert\kappa\vert=k} \|\partial^\kappa f\|_{L^p_{\alpha+pk}}
     +\sum_{\vert\kappa\vert\le k-1} \vert\partial^\kappa f(0)\vert\biggr)
\end{equation}

We do this through a sequence of lemmas.
We first recall Hardy's inequality:
Let $1\le p<\infty$, $\alpha>-1$ and $u<1$.
For every $g\ge 0$ defined on the interval $(u,1)$,
\begin{equation}\label{HardyI}
\int_u^1\Bigl(\int_u^x g(t)\,dt\Bigr)^p(1-x)^\alpha\,dx
\le \Bigl(\frac{p}{\alpha+1}\Bigr)^p\int_u^1 g^p(x)(1-x)^{\alpha+p}\,dx.
\end{equation}
This follows from \cite[Lemma 3.14(ii)]{SW} after a change of variable.

\begin{lemma}\label{LAHardy}
Let $\alpha>-1$, $1\le p<\infty$ and $0<r_0<1$.
There exists a constant $C=C(n,\alpha,p,r_0)$ such that for all $f\in C^1(\mathbb B)$,
\begin{equation*}
\int_{\mathbb B} \vert f(x)\vert^p(1-\vert x\vert^2)^\alpha d\nu(x)
\le C\Bigl(
\int_{\mathbb B} \vert \nabla f(x)\vert^p(1-\vert x\vert^2)^{\alpha+p} d\nu(x)
+\int_{r_0\mathbb B} \vert f(x)\vert^p d\nu(x)
\Bigr).
\end{equation*}
\end{lemma}

\begin{proof}
Throughout the proof, we suppress constants depending only on $n,\alpha,p,r_0$.
Let $\zeta\in\mathbb S$ and $r_0/2\le u\le r_0$.
Since $f(r\zeta)=\int_u^r \frac{d}{dt}f(t\zeta)\,dt+f(u\zeta)$, we have
\begin{equation*}
\vert f(r\zeta)\vert^p
\lesssim \Bigl(\int_u^r\vert\nabla f(t\zeta)\vert\,dt\Bigr)^p+\vert f(u\zeta)\vert^p,
\end{equation*}
and by Hardy's inequality \eqref{HardyI},
\begin{align*}
\int_u^1 nr^{n-1}(1-r^2)^\alpha\vert f(r\zeta)\vert^p\,dr
&\lesssim \int_u^1\Bigl(\int_u^r\vert\nabla f(t\zeta)\vert\,dt\Bigr)^p
(1-r)^\alpha\,dr+\vert f(u\zeta)\vert^p\\
&\lesssim \int_u^1 \vert\nabla f(r\zeta)\vert^p (1-r^2)^{\alpha+p}\,dr
+\vert f(u\zeta)\vert^p\\
&\lesssim \int_u^1 nr^{n-1}\vert\nabla f(r\zeta)\vert^p (1-r^2)^{\alpha+p}\,dr
+\vert f(u\zeta)\vert^p,
\end{align*}
where in the last inequality we use $r\ge u\ge r_0/2$.
Integrating over $\mathbb S$ yields
\begin{equation*}
\int_{\mathbb B\backslash u\mathbb B}
\vert f(x)\vert^p(1-\vert x\vert^2)^\alpha d\nu(x)
\lesssim \int_{\mathbb B\backslash u\mathbb B}
\vert\nabla f(x)\vert^p(1-\vert x\vert^2)^{\alpha+p} d\nu(x)
+\int_{\mathbb S}\vert f(u\zeta)\vert^p d\sigma(\zeta)
\end{equation*}
for all $u\in [r_0/2,r_0]$ with the suppressed constant not depending on $u$.
Thus,
\begin{equation*}
\int_{\mathbb B\backslash r_0\mathbb B}
\vert f(x)\vert^p(1-\vert x\vert^2)^\alpha d\nu(x)
\lesssim \int_{\mathbb B}
\vert\nabla f(x)\vert^p(1-\vert x\vert^2)^{\alpha+p} d\nu(x)
+\int_{\mathbb S}\vert f(u\zeta)\vert^p d\sigma(\zeta)
\end{equation*}
for all $u\in [r_0/2,r_0]$, and multiplying both sides of the inequality
by $nu^{n-1}$ and integrating with respect to $u$ over $[r_0/2,r_0]$,
we deduce
\begin{align*}
\int_{\mathbb B\backslash r_0\mathbb B}
\vert f(x)\vert^p(1-\vert x\vert^2)^\alpha d\nu(x)
\lesssim &\int_{\mathbb B}
\vert\nabla f(x)\vert^p(1-\vert x\vert^2)^{\alpha+p} d\nu(x)\\
&+\int_{r_0\mathbb B\backslash \frac{1}{2}r_0\mathbb B}
\vert f(x)\vert^p d\nu(x).
\end{align*}
Finally, adding
$\int_{r_0\mathbb B}\vert f(x)\vert^p(1-\vert x\vert^2)^\alpha d\nu(x)
\sim \int_{r_0\mathbb B}\vert f(x)\vert^p d\nu(x)$
to both sides gives the desired result.
\end{proof}

For $f\in C^1(\mathbb B)$ and $0<r<1$, let
$\bar{f}_r=\frac{1}{\nu(r\mathbb B)}\int_{r\mathbb B}f(x)\,d\nu(x)$
be the average of $f$ over $r\mathbb B$.
We show that if $\bar{f}_{r_0}=0$, then the last term in
Lemma \ref{LAHardy} can be deleted.

\begin{lemma}\label{Lflessnabla}
Let $\alpha>-1$, $1\le p<\infty$ and $0<r_0<1$.
There exists a constant $C=C(n,\alpha,p,r_0)$ such that for all
$f\in C^1(\mathbb B)$ with $\bar{f}_{r_0}=0$,
\begin{equation*}
\int_{\mathbb B} \vert f(x)\vert^p(1-\vert x\vert^2)^\alpha\, d\nu(x)
\le C
\int_{\mathbb B} \vert \nabla f(x)\vert^p(1-\vert x\vert^2)^{\alpha+p}\, d\nu(x).
\end{equation*}
\end{lemma}

\begin{proof}
By the Poincare inequality, there exists a constant $C=C(n,p,r_0)$ such that
\begin{equation*}
\int_{r_0\mathbb B}\vert f(x)-\bar{f}_{r_0}\vert^p\,d\nu(x)
\le C \int_{r_0\mathbb B}\vert\nabla f(x)\vert^p\,d\nu(x)
\end{equation*}
for all $f\in C^1(\mathbb B)$.
Thus, if $\bar{f}_{r_0}=0$,
\begin{equation*}
\int_{r_0\mathbb B} \vert f(x)\vert^p\, d\nu(x)
\lesssim \int_{r_0\mathbb B} \vert \nabla f(x)\vert^p\, d\nu(x)
\lesssim \int_{r_0\mathbb B} \vert \nabla f(x)\vert^p(1-\vert x\vert^2)^{\alpha+p}\, d\nu(x)
\end{equation*}
and the desired result follows from Lemma \ref{LAHardy}.
\end{proof}

The above lemmas are for $f\in C^1(\mathbb B)$.
We next turn to $\mathcal H$-harmonic functions.
We first consider the special case
$f(x)=\sum_{m=k}^\infty S_m(\vert x\vert)f_m(x)$, the series starting from $k$,
and show that the average $\overline{(\partial^\lambda f)}_r=0$ for all $0<r<1$ and
$\vert\lambda\vert\le k-1$.
\begin{lemma}\label{Laverage0}
Let $k\ge 1$.
If $f\in\mathcal H(\mathbb B)$ with
$f(x)=\sum_{m=k}^\infty S_m(\vert x\vert)f_m(x)$, the series starting from $k$,
then for every multi-index $\lambda$ with $\vert\lambda\vert\le k-1$
and $0<r_0<1$,
\begin{equation*}
  \int_{r_0\mathbb B}\partial^\lambda f(x)\, d\nu(x)=0.
\end{equation*}
\end{lemma}

\begin{proof}
In the proof, we use the fact that if $q\in H_m$ and $p$ is a polynomial
of degree less than $m$, then $\int_{\mathbb S} pq\,d\sigma=0$
(see \cite[Proposition 5.9]{ABR}).

Let $\lambda$ be a multi-index with $\vert\lambda\vert\le k-1$.
By uniform convergence of the series expansion of $f$ and its partial
derivatives, we can differentiate term by term.
We do not need the exact formula of the derivative
$\partial^\lambda(S_m(\vert x\vert)f_m(x))$.
Only terms of special form will occur in the derivative and we only mention these.
Because $S_m(\vert x\vert)=C\,{}_2F_1(m,1-n/2;m+n/2;\vert x\vert^2)$,
for any multi-index $\gamma$, $\partial^\gamma S_m(\vert x\vert)$ is a
sum of multiples of terms of the form
\begin{equation*}
x^\delta G(\vert x\vert^2)
\end{equation*}
with
$\vert\delta\vert \le\vert\gamma\vert$ and $G$ is some function which
depends on $n,m,\gamma,\delta$.
Therefore, $\partial^\lambda(S_m(\vert x\vert)f_m(x))$ is a sum of terms
of the form
\begin{equation*}
x^\delta G(\vert x\vert^2)\partial^\epsilon f_m(x),
\end{equation*}
with $\vert\delta\vert+\vert\epsilon\vert \le \vert\lambda\vert\le k-1$.
Integrating in polar coordinates shows
\begin{equation*}
\int_{r_0\mathbb B} x^\delta G(\vert x\vert^2)\partial^\epsilon f_m(x)\,d\nu(x)
=\int_0^{r_0} nr^{n-1} r^{\vert\delta\vert}G(r^2)r^{m-\vert\epsilon\vert}
\int_{\mathbb S} \zeta^\delta\partial^\epsilon f_m(\zeta)\,d\sigma(\zeta)\,dr.
\end{equation*}
For $m\ge k$, the integral over $\mathbb S$ vanishes since
$\partial^\epsilon f_m\in H_{m-\vert\epsilon\vert}$ and
$\vert\delta\vert<m-\vert\epsilon\vert$.
\end{proof}

\begin{lemma}\label{LStartk}
Let $\alpha>-1$, $1\le p<\infty$ and $k\ge 1$.
There exists a constant $C=C(n,\alpha,p,k)$ such that for all
$f\in\mathcal H(\mathbb B)$ with
$f(x)=\sum_{m=k}^\infty S_m(\vert x\vert)f_m(x)$,
the series starting from $k$, we have
\begin{equation*}
\|f\|_{\mathcal B^p_\alpha}
\le C \sum_{\vert\kappa\vert=k} \|\partial^\kappa f\|_{L^p_{\alpha+pk}}.
\end{equation*}
\end{lemma}

\begin{proof}
Pick an $r_0$ with $0<r_0<1$, say $r_0=1/2$.
Let $\lambda$ be a multi-index with $\vert\lambda\vert=k-1$.
Then, by Lemmas \ref{Lflessnabla} and \ref{Laverage0},
\begin{equation*}
\int_{\mathbb B} \vert\partial^\lambda f(x)\vert^p
(1-\vert x\vert^2)^{\alpha+p(k-1)}\,d\nu(x)
\lesssim
\sum_{\vert\kappa\vert=k}
\int_{\mathbb B} \vert\partial^\kappa f(x)\vert^p
(1-\vert x\vert^2)^{\alpha+pk}\,d\nu(x),
\end{equation*}
that is,
$\|\partial^\lambda f\|_{L^p_{\alpha+p(k-1)}}
\lesssim\sum_{\vert\kappa\vert=k}\|\partial^\kappa f\|_{L^p_{\alpha+pk}}$.
Repeating this for multi-indices with $\vert\lambda\vert=k-2,k-3,\dots,0$,
we obtain the desired result.
\end{proof}

We remove the restriction on $f$ with the following two elementary lemmas.

\begin{lemma}\label{LDerat0}
Let $m\ge 1$ and $p_m\in H_m$.
Then, for every multi-index $\lambda$ with $\vert\lambda\vert<m$, we have
$\partial^\lambda\bigl(S_m(\vert x\vert)p_m(x)\bigr)_{\vert_{x=0}}=0$.
\end{lemma}

\begin{proof}
We have
$\partial^\lambda(S_m(\vert x\vert)p_m(x))
=\sum_{\gamma\le\lambda}\binom{\lambda}{\gamma}
\partial^{\lambda-\gamma} S_m(\vert x\vert)\partial^\gamma p_m(x)$,
and $\partial^\gamma p_m$ is a homogeneous polynomial of degree
$m-\vert\gamma\vert$.
Because $m-\vert\gamma\vert\ge 1$, we have $\partial^\gamma p_m(0)=0$
and the result follows.
\end{proof}

For $k\ge 0$,
\begin{equation*}
\mathcal P_k=\Bigl\{\sum_{m=0}^k S_m(\vert x\vert)p_m(x):p_m\in H_m\Bigr\}
\end{equation*}
is a finite dimensional vector space and
\begin{equation*}
\|p\|_0:=\sum_{\vert\lambda\vert\le k}\vert\partial^\lambda p(0)\vert
\qquad (p\in\mathcal P_k)
\end{equation*}
is a norm on $\mathcal P_k$.
The only non-trivial part to check is $\|p\|_0=0$ implies $p\equiv 0$.

\begin{lemma}\label{Lp0norm}
If $p\in\mathcal P_k$ and $\|p\|_0=0$, then $p\equiv0$.
\end{lemma}

\begin{proof}
Suppose $p(x)=\sum_{m=0}^k S_m(\vert x\vert)p_m(x)$ and $\|p\|_0=0$.
We use induction to show that each $p_m=0$.
First, $p(0)=0$ implies $p_0=0$.
Next, assume that $p_0,p_1,\dots,p_{j-1}=0$ $(j\le k)$.
Let $\lambda$ be a multi-index with $\vert\lambda\vert=j$.
Then, by the induction hypothesis and Lemma \ref{LDerat0},
\begin{align*}
0=\partial^\lambda p(0)
&=\partial^\lambda\Bigl(\sum_{m=0}^k S_m(\vert x\vert)p_m(x)\Bigr)_{\vert_{x=0}}
=\partial^\lambda\bigl(S_j(\vert x\vert)p_j(x)\bigr)_{\vert_{x=0}}\\
&=\sum_{\gamma\le\lambda}\binom{\lambda}{\gamma}
\bigl(\partial^{\lambda-\gamma}S_j(\vert x\vert)\partial^\gamma p_j(x)\bigr)_{\vert_{x=0}}
=S_j(0)\partial^\lambda p_j(0),
\end{align*}
since $\partial^\gamma p_j(0)=0$ for $\vert\gamma\vert\le\vert\lambda\vert-1$.
Because $S_j(0)\neq 0$, we see that $\partial^\lambda p_j(0)=0$, and since this
is true for all $\vert\lambda\vert=j$, we conclude that $p_j=0$.
\end{proof}

\begin{proof}[Proof of Theorem \ref{TBergman} (c)$\Rightarrow$(a)]
Let $f\in\mathcal H(\mathbb B)$ with
$f(x)=\sum_{m=0}^\infty S_m(\vert x\vert)f_m(x)
=:\sum_{m=k}^\infty S_m(\vert x\vert)f_m(x)+p_{k-1}(x)$.
Applying Lemma \ref{LStartk} to $f-p_{k-1}$ shows
\begin{equation*}
\|f\|_{\mathcal B^p_\alpha}
\lesssim \sum_{\vert\kappa\vert=k} \|\partial^\kappa f\|_{L^p_{\alpha+pk}}
+\|p_{k-1}\|_{L^p_\alpha}
+\sum_{\vert\kappa\vert=k} \|\partial^\kappa p_{k-1}\|_{L^p_{\alpha+pk}}.
\end{equation*}
Now, $\|p\|_{L^p_\alpha}
+\sum_{\vert\kappa\vert=k} \|\partial^\kappa p\|_{L^p_{\alpha+pk}}$
is a norm on the finite dimensional vector space $\mathcal P_{k-1}$ and so
it is equivalent to $\|p\|_0$.
Thus,
\begin{equation*}
\|f\|_{\mathcal B^p_\alpha}
\lesssim \sum_{\vert\kappa\vert=k} \|\partial^\kappa f\|_{L^p_{\alpha+pk}}
+\sum_{\vert\kappa\vert\le k-1}\vert\partial^\kappa p_{k-1}(0)\vert.
\end{equation*}
Finally, since $\partial^\kappa p_{k-1}(0)=\partial^\kappa f(0)$
for $\vert\kappa\vert\le k-1$ by Lemma \ref{LDerat0}, we obtain \eqref{dimpa}.
\end{proof}

We now consider the case $\alpha\in\mathbb R$ and show part (a)$\Leftrightarrow$(b)
of Theorem \ref{TBesovPartial}.

\begin{proof}[Proof of Theorem \ref{TBesovPartial} (a)$\Rightarrow$(b)]
The proof is similar to the Bergman case, except that
restriction \eqref{Condalpha} is required to ensure that the conditions
of Lemma \ref{LEbdd} hold.

Suppose \eqref{Condalpha} holds and $f\in\mathcal B^p_\alpha$.
Let $k\ge 1$ satisfy
\begin{equation}\label{alphapk}
\alpha+pk>-1.
\end{equation}
We first show that there exists $C=C(n,\alpha,p,k)$ such that for every
$\vert\kappa\vert=k$,
\begin{equation}\label{partialkappaf}
\|\partial^\kappa f\|_{L^p_{\alpha+pk}}\le C\|f\|_{B^p_\alpha}.
\end{equation}

Pick $t\in\mathbb R$ with $\alpha+pt>-1$ and large enough
$\beta$ satisfying both inequalities
\begin{equation}\label{betabig}
\alpha+1<p(\beta+1)
\qquad \text{and}
\qquad \beta>-1.
\end{equation}
Then the integral representation \eqref{IntRep} holds and differentiating shows
\begin{equation}\label{partialfsmall}
\vert\partial^\kappa f(x)\vert
\lesssim\int_{\mathbb B} \vert\partial^\kappa\mathcal R_\beta(x,y)\vert\,
\bigl((1-\vert y\vert^2)^t \vert D^t_\beta f(y)\vert\bigr)\,d\nu_\beta(y).
\end{equation}
We next estimate $\vert\partial^\kappa\mathcal R_\beta\vert$ with
Theorem \ref{TPartialKernel} in three cases.
Note first that we have $(1-\vert x\vert^2)^t D^t_\beta f\in L^p_\alpha$ and
$\|(1-\vert x\vert^2)^t D^t_\beta f\|_{L^p_\alpha}\sim\|f\|_{\mathcal B^p_\alpha}$.

If $k\le n-2$, then by part (i) of Theorem \ref{TPartialKernel} (note that
we have $\beta>-1$)
\begin{equation*}
(1-\vert x\vert^2)^k\vert\partial^\kappa f(x)\vert
\lesssim (1-\vert x\vert^2)^k\int_{\mathbb B}
\frac{(1-\vert y\vert^2)^t \vert D^t_\beta f(y)\vert}{[x,y]^{n+\beta+k}}
\,d\nu_\beta(y).
\end{equation*}
The right-hand side is $E_{\beta,k}((1-\vert y\vert^2)^t \vert D^t_\beta f(y)\vert)$
and $E_{\beta,k}\colon L^p_\alpha\to L^p_\alpha$ is bounded by Lemma \ref{LEbdd}
because of \eqref{alphapk} and \eqref{betabig}.
This shows \eqref{partialkappaf}.

We need \eqref{Condalpha} when $k\ge n-1$.
If $k=n-1$, we choose $0<\varepsilon<1$ satisfying $-p(n-1-\varepsilon)<\alpha+1$
possible by \eqref{Condalpha}.
Then part (b) of Theorem \ref{TPartialKernel} (ii) implies

\begin{equation*}
(1-\vert x\vert^2)^{n-1}\vert\partial^\kappa f(x)\vert
\lesssim (1-\vert x\vert^2)^{n-1-\varepsilon}\int_{\mathbb B}
\frac{(1-\vert y\vert^2)^t \vert D^t_\beta f(y)\vert}{[x,y]^{n+\beta+n-1-\varepsilon}}
\,d\nu_\beta(y).
\end{equation*}
The conditions of Lemma \ref{LEbdd} are satisfied and \eqref{partialkappaf} follows.

If $k\ge n$, then by part (iii) of Theorem \ref{TPartialKernel}, we have
\begin{equation*}
(1-\vert x\vert^2)^k\vert\partial^\kappa f(x)\vert
\lesssim (1-\vert x\vert^2)^{n-1}\int_{\mathbb B}
\frac{(1-\vert y\vert^2)^t \vert D^t_\beta f(y)\vert}{[x,y]^{n+\beta+n-1}}
\,d\nu_\beta(y),
\end{equation*}
and conditions of Lemma \ref{LEbdd} are satisfied by \eqref{Condalpha}.
Thus, \eqref{partialkappaf} holds.

Finally, to see
$\sum_{\vert\kappa\vert\le k-1}\vert\partial^\kappa f(0)\vert
\lesssim \|f\|_{\mathcal B^p_\alpha},$
put $x=0$ in \eqref{partialfsmall}.
As $\vert\partial^\kappa\mathcal R_\beta(0,y)\vert\lesssim 1$,
we obtain
$\vert\partial^\kappa f(0)\vert\lesssim \|D^t_\beta f\|_{L^1_{\beta+t}}$,
and the result follows from \eqref{DtinL1}.
\end{proof}

\begin{proof}[Proof of Theorem \ref{TBesovPartial} (b)$\Rightarrow$(a)]
Suppose $\alpha\in\mathbb R$, $f\in\mathcal H(\mathbb B)$
and there exists $k\ge 1$ with $\alpha+kp>-1$ such that
$\partial^\kappa f\in L^p_{\alpha+pk}$ for every multi-index $\vert\kappa\vert=k$.
We show that $f\in\mathcal B^p_\alpha$ and there exists a constant
$C=C(n,\alpha,p,k)$ such that
\begin{equation*}
\|f\|_{\mathcal B^p_\alpha}
\le C
\Bigl(\sum_{\vert\kappa\vert=k} \|\partial^\kappa f\|_{L^p_{\alpha+pk}}
     +\sum_{\vert\kappa\vert\le k-1} \vert\partial^\kappa f(0)\vert\Bigr).
\end{equation*}

We employ the tangential derivative characterization of $\mathcal B^p_\alpha$.
Because $\|f\|_{\mathcal B^p_\alpha}
\sim\sum_{T^k\in\mathcal T^k}\|T^k f\|_{L^p_{\alpha+pk}}+\vert f(0)\vert$
and $\vert T^k f(x)\vert
\le C\sum_{1\le\vert\lambda\vert\le k}\vert\partial^\lambda f(x)\vert$
for $T^k\in\mathcal T^k$, it suffices to show that
\begin{equation}\label{partlambdaf}
\vert f(0)\vert
+\sum_{1\le\vert\lambda\vert\le k}\|\partial^\lambda f\|_{L^p_{\alpha+pk}}
\lesssim \sum_{\vert\kappa\vert=k} \|\partial^\kappa f\|_{L^p_{\alpha+pk}}
     +\sum_{\vert\kappa\vert\le k-1} \vert\partial^\kappa f(0)\vert.
\end{equation}

Assume first that
$f(x)=\sum_{m=k}^\infty S_m(\vert x\vert)f_m(x)$ with the series starting from $k$.
Let $\lambda$ be a multi-index with $\vert\lambda\vert=k-1$.
Then by Lemmas \ref{Lflessnabla} and \ref{Laverage0} and the fact that
$(1-\vert x\vert^2)^{\alpha+kp+p}\le (1-\vert x\vert^2)^{\alpha+kp}$,
we see that
\begin{equation*}
\int_{\mathbb B} \vert \partial^\lambda f(x)\vert^p(1-\vert x\vert^2)^{\alpha+pk}\, d\nu(x)
\lesssim\sum_{\vert\kappa\vert=k}
\int_{\mathbb B} \vert \partial^\kappa f(x)\vert^p(1-\vert x\vert^2)^{\alpha+pk}\, d\nu(x).
\end{equation*}
Repeating this for $\vert\lambda\vert=k-2,\dots,1$ shows
\begin{equation}\label{FirstEst}
\sum_{1\le\vert\lambda\vert\le k-1}\|\partial^\lambda f\|_{L^p_{\alpha+pk}}
\lesssim
\sum_{\vert\kappa\vert=k}\|\partial^\kappa f\|_{L^p_{\alpha+pk}}.
\end{equation}

Next, we remove the assumption on $f$.
If $f(x)=\sum_{m=0}^\infty S_m(\vert x\vert)f_m(x)
=:\sum_{m=k}^\infty S_m(\vert x\vert)f_m(x)+p_{k-1}(x)$,
then \eqref{FirstEst} holds for $f-p_{k-1}$ and so,
\begin{equation*}
\sum_{1\le\vert\lambda\vert\le k-1}\|\partial^\lambda f\|_{L^p_{\alpha+pk}}
\lesssim
\sum_{\vert\kappa\vert=k}\|\partial^\kappa f\|_{L^p_{\alpha+pk}}
+\sum_{1\le\vert\lambda\vert\le k}\|\partial^\lambda p_{k-1}\|_{L^p_{\alpha+pk}}.
\end{equation*}
Now,
$\sum_{1\le \vert\lambda\vert\le k}\|\partial^\lambda p_{k-1}\|_{L^p_{\alpha+pk}}
+\vert p_{k-1}(0)\vert$
is a norm on the finite dimensional vector space $\mathcal P_{k-1}$ and is equivalent
to the norm
$\|p_{k-1}\|_0=\sum_{\vert\lambda\vert\le k-1}\vert\partial^\lambda p_{k-1}(0)\vert$.
Finally, since $\partial^\lambda p_{k-1}(0)=\partial^\lambda f(0)$
for $\vert\lambda\vert\le k-1$ by Lemma \ref{LDerat0}, we obtain \eqref{partlambdaf}.
\end{proof}

\section{Normal Derivative Characterization}

In this section, we prove parts (a)$\Leftrightarrow$(d) of Theorem \ref{TBergman}
and (a)$\Leftrightarrow$(c) of Theorem \ref{TBesov}.
The proof of part (a)$\Rightarrow$(d) of Theorem \ref{TBergman} is same as the
proof of part (a)$\Rightarrow$(c).
The only difference is, in \eqref{reprform}, instead of partial
derivatives we take normal derivatives, and afterwards
refer to Theorem \ref{TNormalKernel} instead of Theorem \ref{TPartialKernel}.
Similarly, proof of (a)$\Rightarrow$(c) of Theorem \ref{TBesov} is the same
as the proof of (a)$\Rightarrow$(b), and is omitted.

We now show part (d)$\Rightarrow$(a) of Theorem \ref{TBergman}.
For this, we will use two lemmas.
First one is a mean-value inequality of Grellier-Jaming
for partial derivatives of $N^k f$ of an $\mathcal H$-harmonic function $f$
(see \cite[Lemma 3.3]{GJ}).
Below, $B(a,r)$ denotes the Euclidean ball of radius $r$ centered at $a$.

\begin{lemma}\label{LGJ}
Let $0<\varepsilon<1$, $0<p<\infty$ and $k,d\in\mathbb N$.
There exists a constant $C=C(n,\varepsilon,p,k,d)$ such that
for all $x\in\mathbb B$, $f\in\mathcal H(\mathbb B)$ and
multi-indices $\lambda$ with $\vert\lambda\vert\le d$,
we have
\begin{equation*}
\vert\partial^\lambda N^k f(x)\vert^p
\le \frac{C}{(1-\vert x\vert)^{n+pd}}
\int_{B(x,(1-\vert x\vert)\varepsilon)} \vert N^k f(y)\vert^p\, d\nu(y).
\end{equation*}
\end{lemma}

\begin{lemma}\label{LNormal}
Let $\alpha>-1$ and $1\leq p<\infty$.
There exists $C=C(n,\alpha,p)$ such that for all $g\in C^2(\mathbb B)$,
we have
\begin{equation}\label{glessN}
\|g-g(0)\|_{L^p_{\alpha}}
\le C
\Bigl(\|Ng\|_{L^p_{\alpha+p}}
+\sup_{\vert x\vert\le 1/2}\vert\nabla Ng(x)\vert\Bigr).
\end{equation}
\end{lemma}

\begin{proof}
Note first that for $0<r<1$ and $\zeta\in\mathbb S$,
\begin{equation*}
g(r\zeta)-g(0)=\int_0^r\frac{Ng(t\zeta)}{t}\,dt.
\end{equation*}
Now, integrating in polar coordinates, using above, and then Hardy's
inequality \eqref{HardyI} shows that
\begin{align*}
\|g-g(0)\|^p_{L^p_\alpha}
&=\int_{\mathbb S}\int_0^1 nr^{n-1}(1-r^2)^\alpha\vert g(r\zeta)-g(0)\vert^p
\,dr\,d\sigma(\zeta)\\
&\lesssim \int_{\mathbb S}\int_0^1 (1-r)^\alpha
\Bigl(\int_0^r\frac{\vert Ng(t\zeta)\vert}{t}\,dt\Bigr)^p
\,dr\,d\sigma(\zeta)\\
&\lesssim\int_{\mathbb S}\int_0^1 (1-r)^{\alpha+p}
\frac{\vert Ng(r\zeta)\vert^p}{r^p}\,dr\,d\sigma(\zeta)\\
&\lesssim \int_{\mathbb B} \frac{\vert Ng(x)\vert^p}{\vert x\vert^{p+n-1}}
(1-\vert x\vert^2)^{\alpha+p}\, d\nu(x)\\
&=\int_{\frac{1}{2}\mathbb B} \frac{\vert Ng(x)\vert^p}{\vert x\vert^{p+n-1}}
\,d\nu_{\alpha+p}(x)
+\int_{\mathbb B\backslash\frac{1}{2}\mathbb B}
\frac{\vert Ng(x)\vert^p}{\vert x\vert^{p+n-1}}\,d\nu_{\alpha+p}(x)
=:I_1+I_2.
\end{align*}
In the second integral we have $\vert x\vert\ge 1/2$ and so,
$I_2\lesssim \|Ng\|^p_{L^p_{\alpha+p}}$.
For the first integral, observe that for $\vert x\vert \le 1/2$,
by Lagrange's mean-value inequality,
\begin{equation*}
\vert Ng(x)\vert\le \vert x\vert\sup_{\vert y\vert\le 1/2}\vert\nabla Ng(y)\vert,
\end{equation*}
and therefore,
\begin{equation*}
I_1\le \sup_{\vert y\vert\le 1/2}\vert\nabla Ng(y)\vert^p
\int_0^{1/2}nr^{n-1}\frac{r^p}{r^{p+n-1}}(1-r^2)^{\alpha+p}\,dr
\lesssim \sup_{\vert y\vert\le 1/2}\vert\nabla Ng(y)\vert^p.
\qedhere
\end{equation*}
\end{proof}

\begin{proof}[Proof of Theorem \ref{TBergman} (d)$\Rightarrow$(a)]
Suppose $f\in\mathcal H(\mathbb B)$ and $N^k f\in L^p_{\alpha+pk}$ for some $k\ge 1$.
We will apply Lemma \ref{LNormal} with $g=N^{k-1}f$.
First, we estimate the last term in \eqref{glessN}.
Lemma \ref{LGJ} with $\varepsilon=1/4$ shows that there
exists $C=C(n,p,k)$ such that
\begin{equation*}
\vert\nabla N^k f(x)\vert^p
\le \frac{C}{(1-\vert x\vert)^{n+p}}
\int_{B(x,(1-\vert x\vert)/4)} \vert N^k f(y)\vert^p\, d\nu(y).
\end{equation*}
For $\vert x\vert\le 1/2$, we have
$B(x,(1-\vert x\vert)/4)\subset B(0,3/4)$, and therefore,
\begin{equation*}
\sup_{\vert x\vert\le 1/2}\vert\nabla N^{k}f(x)\vert^p
\lesssim \int_{B(0,3/4)}
\vert N^k f(y)\vert^p (1-\vert y\vert^2)^{\alpha+kp}\, d\nu(y)
\lesssim \|N^k f\|^p_{L^p_{\alpha+pk}}.
\end{equation*}
Thus, Lemma \ref{LNormal} with $g=N^{k-1}f$ implies
\begin{equation*}
\|N^{k-1}f\|_{L^p_{\alpha+p(k-1)}}
\lesssim \|N^{k}f\|_{L^p_{\alpha+pk}}.
\end{equation*}
Repeating this argument gives
$\|f-f(0)\|_{\mathcal B^p_\alpha}\lesssim \|N^k f\|_{L^p_{\alpha+pk}}$.
\end{proof}

\begin{proof}[Proof of Theorem \ref{TBesovPartial} (c)$\Rightarrow$(a)]
In this part we do not need the restriction \eqref{Condalpha}.
Suppose that $\alpha\in\mathbb R$, $1\le p<\infty$, $f\in\mathcal H(\mathbb B)$
and there exists $k\in\mathbb N$ with $\alpha+pk>-1$ such that
$N^k f\in L^p_{\alpha+pk}$.
We show that $f\in\mathcal B^p_\alpha$ and there exists $C=C(n,\alpha,p,k)$
such that
\begin{equation}\label{NormfLess}
\|f-f(0)\|_{\mathcal B^p_\alpha}\le C\|N^k f\|_{L^p_{\alpha+pk}}.
\end{equation}
For this we show that
$\|T^k f\|_{L^p_{\alpha+pk}}\le C\|N^k f\|_{L^p_{\alpha+pk}}$
for every $T^k\in\mathcal T^k$, which by Defn.~\ref{DBesov}
implies $f\in\mathcal B^p_\alpha$ and
$\|f-f(0)\|_{\mathcal B^p_\alpha}
\sim\sum_{T^k\in\mathcal T^k}\|T^k f\|_{L^p_{\alpha+pk}}$
satisfies \eqref{NormfLess}.

First, applying Lemma \ref{LGJ} with $\varepsilon=1/3$ and $d=k$,
we see that there exists $C=C(n,p,k)$ such that for all
multi-indices $\lambda$ with $\vert\lambda\vert\le k$ and
$x\in\mathbb B$, we have
\begin{equation*}
\bigl\vert\partial^\lambda N^k f(x)\bigr\vert^p
\le \frac{C}{(1-\vert x\vert)^{n+pk}}
\int_{B(x,(1-\vert x\vert)/3)} \vert N^k f(y)\vert^p\,d\nu(y).
\end{equation*}
Thus, for all $\vert\lambda\vert\le k$,
\begin{equation*}
\|\partial^\lambda N^k f\|^p_{L^p_{\alpha+2pk}}
\lesssim \int_{\mathbb B} \int_{B(x,(1-\vert x\vert)/3)}
\vert N^k f(y)\vert^p\,d\nu(y)(1-\vert x\vert^2)^{\alpha+pk-n}\,d\nu(x).
\end{equation*}
Note that if $y\in B(x,(1-\vert x\vert)/3)$, then $x\in B(y,(1-\vert y\vert)/2)$.
Therefore, changing the order of integration, we deduce
\begin{equation*}
\|\partial^\lambda N^k f\|^p_{L^p_{\alpha+2pk}}
\lesssim \int_{\mathbb B}\vert N^k f(y)\vert^p
\int_{B(y,(1-\vert y\vert)/2)}(1-\vert x\vert^2)^{\alpha+pk-n}\,d\nu(x)\,d\nu(y).
\end{equation*}
Because
$1-\vert x\vert^2\sim1-\vert y\vert^2$ for $x\in B(y,(1-\vert y\vert)/2)$
and $\nu(B(y,(1-\vert y\vert)/2))\sim (1-\vert y\vert^2)^n$, we conclude that
\begin{equation*}
\|\partial^\lambda N^k f\|_{L^p_{\alpha+2pk}}
\lesssim \|N^k f\|_{L^p_{\alpha+pk}}
\end{equation*}
for all $\vert\lambda\vert\le k$.
Since
$\vert T^k g(x)\vert
\le C\sum_{1\le\vert\lambda\vert\le k}\vert\partial^\lambda g(x)\vert$
for $x\in\mathbb B$, this implies
\begin{equation*}
\|T^k N^k f\|_{L^p_{\alpha+2pk}}
\lesssim \|N^k f\|_{L^p_{\alpha+pk}},
\end{equation*}
and noting that $T^k N^k=N^k T^k$ because $N$ commutes with every $T_{i,j}$,
we obtain
\begin{equation*}
\|N^k(T^k f)\|_{L^p_{\alpha+2pk}}
\lesssim \|N^k f\|_{L^p_{\alpha+pk}}.
\end{equation*}
Finally, since $T^k f\in\mathcal H(\mathbb B)$ and $\alpha+pk>-1$, we are
in the Bergman zone, and part (d)$\Rightarrow$(a) of Theorem \ref{TBergman}
implies $T^k f\in\mathcal B^p_{\alpha+pk}$ and
\begin{equation*}
\|T^k f\|_{L^p_{\alpha+pk}}
\lesssim \|N^k(T^k f)\|_{L^p_{\alpha+2pk}}
\lesssim \|N^k f\|_{L^p_{\alpha+pk}}.
\end{equation*}
This finishes the proof.
\end{proof}

\end{document}